\documentclass[a4paper, 11pt]{amsart}
\usepackage{amsmath, amssymb, amsthm, bm, eucal}
\usepackage{braket, stmaryrd,mathrsfs}
\usepackage{cite}
\usepackage[dvipdfmx]{graphicx,xcolor}

\calclayout

\newcommand{\pr}[1]{#1^{\prime}}

\newcommand{\del}{\partial}

\newcommand{\sgn}{\mathrm{sgn}}
\newcommand{\mfrak}[1]{\mathfrak{#1}}
\newcommand{\mcal}[1]{\mathcal{#1}}
\newcommand{\mbb}[1]{\mathbb{#1}}
\newcommand{\mrm}[1]{\mathrm{#1}}

\newcommand{\scr}[1]{\mathscr{#1}}
\newcommand{\what}[1]{\widehat{#1}}
\newcommand{\no}[1]{:\hspace{-3pt} #1\hspace{-3pt}:\hspace{3pt}}

\makeatletter
\newcommand{\xRightarrow}[2][]{\ext@arrow 0359\Rightarrowfill@{#1}{#2}}
\makeatother

\theoremstyle{plain}
\newtheorem{thm}{Theorem}[section]
\newtheorem{lem}[thm]{Lemma}
\newtheorem{prop}[thm]{Proposition}
\newtheorem{cor}[thm]{Corollary}
\theoremstyle{definition}
\newtheorem{defn}[thm]{Definition}
\theoremstyle{remark}
\newtheorem{rem}[thm]{Remark}
\newtheorem{exam}[thm]{Example}

\makeatletter
    
    \@addtoreset{equation}{section}
\makeatother

\title[]{Free field theory and observables of periodic Macdonald processes}
\author{Shinji Koshida}
\address{Department of Physics, Faculty of Science and Engineering, Chuo University, Kasuga, Bunkyo, Tokyo 112-8551, Japan}
\email{koshida@phys.chuo-u.ac.jp}

\begin{document}

\begin{abstract}
We propose periodic Macdonald processes as a $(q,t)$-deformation of periodic Schur processes and a periodic analogue of Macdonald processes. It is known that, in the theory of stochastic processes related to a family of symmetric functions, the Cauchy-like identity gives an explicit expression of a partition function. We compute the partition functions of periodic Macdonald processes relying on the free field realization of the Macdonald theory. We also study several families of observables for periodic Macdonald processes and give formulas of their moments. The technical tool is the free field realization of operators that are diagonalized by the Macdonald symmetric functions, where the operators admit expressions by means of vertex operators.
We show that, when we adopt Plancherel specializations, the corresponding periodic Macdonald process is related to a Young diagram-valued periodic continuous process.
It is known that, for periodic Schur processes, determinantal formulas are only available when we extend them to take into account the charge. We also consider this kind of shift-mixed periodic Macdonald processes and discuss their Schur-limit.
\end{abstract}


\maketitle


\section{Introduction}
\label{sect:introduction}
Recent studies in the field of Integrable Probability revealed that many integrable stochastic models are unified in the framework of Macdonald processes \cite{BorodinCorwin2014,BorodinCorwinGorinShakirov2016} in the sense that, at a certain choice of specializations and parameters, the probability law of the corresponding Macdonald process coincides with that of a stochastic model of interest. Examples of such models include $q$-TASEP \cite{BorodinCorwin2014}, Hall--Littlewood plane partitions \cite{Dimitrov2018}, general $\beta$-ensembles \cite{BorodinGorin2015} (see also \cite{ Petrov2009,BorodinGorin2012,Oconnell2012, OconnellSeppalainenZygouras2014,CorwinOconnellSeppalainenZygouras2014,Borodin2014,BorodinPetrov2014,Corwin2014}).
One of features of Macdonald processes is its determinantal structure; there are several series of observables available for Macdonald processes whose expectation values are written as multiple integrals of determinants, and moreover, when we consider the generating function of a series of observables, its expectation value is expressed by a Fredholm determinant. Therefore, once a stochastic model is identified with a Macdonald process, one can study it by means of its determinantal structure and even carry out an asymptotic analysis (e.g. \cite{BorodinCorwinRemenik2013,BorodinCorwin2014,BorodinCorwin2014b, BorodinCorwinFerrari2014,Barraquand2015, BorodinGorin2015, FerrariVeto2015}).

At the Schur-limit, where the Macdonald symmetric functions reduce to the Schur symmetric functions, Macdonald processes obviously reduce to Schur processes \cite{Okounkov2001,OkounkovReshetikhin2003}. In \cite{Borodin2007}, there was proposed a periodic analogue of Schur processes and shown that the periodic Schur processes admit determinantal structure in the correlation functions after taking into account the charge (see also \cite{BeteaBouttier2019} for an alternative approach). Then, it would be natural to ask how a periodic analogue of Macdonald processes can be formulated and observables are available, which is addressed in the present paper.

We fix parameters $0<q,t,u<1$ and write $P_{\lambda/\mu}(X;q,t)$ and $Q_{\lambda/\mu}(X;q,t)$ for the Macdonald symmetric function and the dual Macdonald symmetric function associated with a skew partition $\lambda/\mu$. Here, we understand $X=(x_{1},x_{2},\dots)$ as a set of infinitely many variables.
We also take sequences $\bm{\rho}^{+}=(\rho^{+}_{0},\dots,\rho^{+}_{N-1})$ and $\bm{\rho}^{-}=(\rho^{-}_{1},\dots,\rho^{-}_{N})$ of $(q,t)$-Macdonald positive specializations of the ring of symmetric functions (see Sect.\ \ref{sect:preliminaries} for definition).
We associate to two sequences $\bm{\lambda}=(\lambda^{1},\dots,\lambda^{N})$ and $\bm{\mu}=(\mu^{1},\dots,\mu^{N})$ of partitions 
the following weight:
\begin{equation}
\label{eq:periodic_weight}
	W^{\bm{\rho}^{+},\bm{\rho}^{-}}_{q,t;u}(\bm{\lambda},\bm{\mu})=u^{|\mu^{N}|}\prod_{i=1}^{N}Q_{\lambda^{i}/\mu^{i}}(\rho^{-}_{i};q,t)P_{\lambda^{i+1}/\mu^{i}}(\rho^{+}_{i};q,t),
\end{equation}
where we understand $\lambda^{N+1}=\lambda^{1}$ and $\rho^{+}_{N}=\rho^{+}_{0}$.
It is obvious that the weight vanishes unless the partitions satisfy the following inclusion properties:
\begin{equation}
	\mu^{N}=\mu^{0}\subset \lambda^{1} \supset \mu^{1} \subset \lambda^{2} \supset \cdots \subset \lambda^{N-1} \supset \mu^{N-1} \subset \lambda^{N}\supset \mu^{N}.
\end{equation}
In this paper, the following $n$-fold $q$-Pochhammer symbol is extensively used:
\begin{equation}
	(a;p_{1},\dots,p_{n})_{\infty}=\prod_{i_{1},\dots, i_{n}=0}^{\infty}(1-ap_{1}^{i_{1}}\cdots p_{n}^{i_{n}}).
\end{equation}

The following proposition gives the partition function of a periodic Macdonald process.
\begin{prop}
\label{prop:Cauchy_identity}
We have
\begin{equation}
	\Pi_{q,t;u}(\bm{\rho}^{+};\bm{\rho}^{-}):=\sum_{\bm{\lambda},\bm{\mu}\in\mbb{Y}^{N}}W^{\bm{\rho}^{+},\bm{\rho}^{-}}_{q,t;u}(\bm{\lambda},\bm{\mu})
	=\frac{1}{(u;u)_{\infty}}\frac{\prod_{i=0}^{N-1}\prod_{j=1}^{N}\tilde{\Pi}_{q,t;u}(\rho^{+}_{i};\rho^{-}_{j})}{\prod_{i=0}^{N-1}\prod_{j=1}^{i}\tilde{\Pi}_{q,t;u}(\rho^{+}_{i};\rho^{-}_{j})},
\end{equation}
where
\begin{equation}
	\tilde{\Pi}_{q,t;u}(X;Y)=\prod_{i,j\ge 1}\frac{(tx_{i}y_{j};q,u)_{\infty}}{(x_{i}y_{j};q,u)_{\infty}}.
\end{equation}
\end{prop}
Though this could be proved by a direct computation, we show it in Sect.\ \ref{sect:periodic_Macdonald} by identifying the ring of symmetric functions with a Fock representation of a Heisenberg algebra and applying free field computation.
The relevant formula is for the trace of a vertex operator (in an extended sense) shown in Subsect.\ \ref{subsect:Fock_representation}.

It is natural to define a periodic Macdonald process as follows:
\begin{defn}
Fix $N\in\mbb{N}$ and parameters $0<q,t,u<1$, take sequences $\bm{\rho}^{+}=(\rho^{+}_{0},\dots,\rho^{+}_{N-1})$ and $\bm{\rho}^{-}=(\rho^{-}_{1},\dots,\rho^{-}_{N})$ of $(q,t)$-Macdonald positive specializations of the ring of symmetric functions. 
We define a probability measure on $\mbb{Y}^{N}$ from the weight in (\ref{eq:periodic_weight}) by
\begin{equation}
	\mbb{P}^{\bm{\rho}^{+},\bm{\rho}^{-}}_{q,t;u}(\lambda^{1},\cdots,\lambda^{N})=\frac{1}{\Pi_{q,t;u}(\bm{\rho}^{+};\bm{\rho}^{-})}\sum_{\bm{\mu}\in\mbb{Y}^{N}}W^{\bm{\rho}^{+},\bm{\rho}^{-}}_{q,t;u}(\bm{\lambda},\bm{\mu})
\end{equation}
and call it an $N$-step periodic Macdonald process. When $N=1$, the probability measure on $\mbb{Y}$ is called a periodic Macdonald measure.
\end{defn}

In Sect.\ \ref{eq:Plancherel_process}, we will define a continuous time periodic stochastic process on $\mbb{Y}$ called a stationary periodic Macdonald--Plancherel process extending the formulation presented in \cite{BeteaBouttier2019}, and show that the marginal law for a fixed sequence of times is described by a periodic Macdonald process specified by Plancherel specializations.
This stochastic process reduces to the one considered in \cite{BeteaBouttier2019} at $q=t$.

We also study observables for periodic Macdonald processes.
In \cite{Shiraishi2006, FeiginHashizumeHoshinoShiraishiYanagida2009}, the authors showed that several operators that are diagonalized by the Macdonald symmetric functions
admit expressions using vertex operators on a Fock space.
In the previous work \cite{SK2019}, we applied their construction to Macdonald processes, recognized that earlier results regarding Macdonald processes are recovered in the free field approach and furthermore found that the determinantal structure of Macdonald processes is manifest at the operator level.
In this paper, we apply the method in \cite{SK2019} to periodic Macdonald processes.

In our setting, any function on $\mbb{Y}$ can be regarded as a random variable. For $N$ random variables $f_{1},\dots, f_{N}$, we define their moment under an $N$-step periodic Macdonald process $\mbb{P}^{\bm{\rho}^{+},\bm{\rho}^{-}}_{q,t;u}$ by
\begin{equation}
	\mbb{E}^{\bm{\rho}^{+},\bm{\rho}^{-}}_{q,t;u}\left[f_{1}[1]\cdots f_{N}[N]\right]:=\sum_{(\lambda^{1},\dots,\lambda^{N})\in\mbb{Y}^{N}}f_{1}(\lambda^{1})\cdots f_{N}(\lambda^{N})\mbb{P}^{\bm{\rho}^{+},\bm{\rho}^{-}}_{q,t;u}(\lambda^{1},\dots,\lambda^{N}).
\end{equation}
If $N=1$, we simply write $\mbb{E}^{\rho^{+},\rho^{-}}_{q,t;u}[f]:=\mbb{E}^{\rho^{+},\rho^{-}}_{q,t;u}[f[1]]$.
In the description of the results, the following function of $\bm{z}^{(a)}=(z^{(a)}_{i},\dots, z^{(a)}_{r_{a}})$, $a=1,\dots, N$ depending on two parameters $p_{1},p_{2}$ is used repeatedly.
\begin{align}
	\Delta_{p_{1},p_{2};u}(\bm{z}^{(1)},\dots, \bm{z}^{(N)})=&\prod_{a<b}\prod_{i=1}^{r_{a}}\prod_{j=1}^{r_{b}}\frac{(z^{(b)}_{j}/z^{(a)}_{i};u)_{\infty}(p_{1}p_{2}z^{(b)}_{j}/z^{(a)}_{i};u)_{\infty}}{(p_{1}z^{(b)}_{j}/z^{(a)}_{i};u)_{\infty}(p_{2}z^{(b)}_{j}/z^{(a)}_{i};u)_{\infty}} \\
	&\times \prod_{a\ge b}\prod_{i=1}^{r_{a}}\prod_{j=1}^{r_{b}}\frac{(uz^{(b)}_{j}/z^{(a)}_{i};u)_{\infty}(p_{1}p_{2}uz^{(b)}_{j}/z^{(a)}_{i};u)_{\infty}}{(p_{1}uz^{(b)}_{j}/z^{(a)}_{i};u)_{\infty}(p_{2}uz^{(b)}_{j}/z^{(a)}_{i};u)_{\infty}}. \notag
\end{align}

Let us fix a certain family of specializations of symmetric functions. The specialization denoted by $q^{\lambda}t^{-\rho+n}$, $\lambda\in\mbb{Y}$, $n\in\mbb{Z}$ is defined by
\begin{equation}
\label{eq:specialization_lambda_rho}
	p_{r}(q^{\lambda}t^{-\rho+n}):=\sum_{i=1}^{\ell (\lambda_{i})}(q^{\lambda_{i}}t^{-i+n})^{r}+\frac{t^{-r(\ell (\lambda)-n+1)}}{1-t^{-r}}, \quad n\in\mbb{N}, \quad \lambda\in\mbb{Y},
\end{equation}
where $p_{r}$, $r\in\mbb{N}$ is the $r$-th power-sum symmetric function.
The specialization $q^{-\lambda}t^{\rho-n}$, $\lambda\in\mbb{Y}$, $n\in\mbb{Z}$ is defined by replacing $q$ by $q^{-1}$ and $t$ by $t^{-1}$ in Eq.~(\ref{eq:specialization_lambda_rho}).

We introduce a series of observables.
For $r\in\mbb{N}$, we consider a random variable $\mcal{E}_{r}:\mbb{Y}\to\mbb{F}$ defined by $\mcal{E}_{r}(\lambda):=e_{r}(q^{\lambda}t^{-\rho+1})$, $\lambda\in\mbb{Y}$, where $e_{r}$ is the $r$-th elementary symmetric function.
Then, the moment of these random variables under an $N$-step periodic Macdonald process is computed as follows.

\begin{thm}
\label{thm:first_series}
For $r_{1},\dots, r_{N}\in\mbb{N}$, the moment of $\mcal{E}_{r_{1}},\dots, \mcal{E}_{r_{N}}$ under the $N$-step periodic Macdonald process $\mbb{P}^{\bm{\rho}^{+},\bm{\rho}^{-}}_{q,t;u}$ is computed as
\begin{align}
\label{eq:expectation_first_series}
	&\mbb{E}^{\bm{\rho}^{+},\bm{\rho}^{-}}_{q,t;u}[\mcal{E}_{r_{1}}[1]\cdots \mcal{E}_{r_{N}}[N]] \\
	&=\frac{1}{\bm{r}!}\int D^{\bm{r}}\bm{z} \prod_{a=1}^{N}\det\left(K^{\bm{\rho}^{+},\bm{\rho}^{-};a}_{q,t;u;\mcal{E}}(z_{i}^{(a)},z_{j}^{(a)})\right)_{1\le i,j\le r_{a}}\Delta_{q,t^{-1};u}(\bm{z}^{(1)},\dots, \bm{z}^{(N)}), \notag
\end{align}
where we set $\bm{r}!:=\prod_{a=1}^{N}r_{a}!$ and
\begin{align}
	K_{q,t;u;\mcal{E}}^{\bm{\rho}^{+},\bm{\rho}^{-};a}(z,w):=\frac{1}{z-t^{-1}w}\prod_{n>0}&\prod_{b=0}^{N-1}\exp\left(\left(u^{n}\chi [b\ge a]+\chi [b<a]\right)\frac{1-t^{-n}}{1-u^{n}}\frac{p_{n}(\rho_{b}^{+})}{n}z^{n}\right) \\
	&\times\prod_{b=1}^{N}\exp\left(-\left(u^{n}\chi [b<a]+\chi [b\ge a]\right)\frac{1-t^{n}}{1-u^{n}}\frac{p_{n}(\rho_{b}^{-})}{n}z^{-n}\right). \notag
\end{align}
Here $\chi[\cdot]$ is the indicator defined by $\chi [P]=1$ if $P$ is true and $\chi [P]=0$ if $P$ is false.
The part $(z-t^{-1}w)^{-1}$ is understood formally as $\sum_{k=0}^{\infty}z^{-1}(t^{-1}w/z)^{k}$, and
the functional $\int D^{\bm{r}}\bm{z}$ takes the coefficient of $\prod_{a=1}^{N}\prod_{i=1}^{r_{a}}(z_{i}^{(a)})^{-1}$.
\end{thm}

We will prove Theorem~\ref{thm:first_series} relying on the free field realization of the observable $\mcal{E}_{r}$, $r\in\mbb{N}$ and the trace formula for a vertex operator in Sect.\ \ref{sect:observables}, where we also see the moments of other series of observables. Note that, in the formula (\ref{eq:expectation_first_series}), the measure part $\Delta_{q,t^{-1};u}(\bm{z}^{(1)},\dots,\bm{z}^{(N)})$ is {\it universal} in the sense that it is independent of specializations $\bm{\rho}^{+}$ and $\bm{\rho}^{-}$.
Now let us see a particular case of $N=1$ as a corollary.
\begin{cor}
\label{cor:expectation_first_series}
For $r\in\mbb{N}$, the expectation value of $\mcal{E}_{r}$ under the periodic Macdonald measure $\mbb{P}^{\rho^{+},\rho^{-}}_{q,t;u}$ is given by
\begin{align}
	\mbb{E}^{\rho^{+},\rho^{-}}_{q,t;u}[\mcal{E}_{r}]=\frac{1}{r!}\int D^{r}\bm{z} \det\left(K^{\rho^{+},\rho^{-}}_{q,t;u;\mcal{E}}(z_{i},z_{j})\right)_{1\le i,j\le r}\Delta_{q,t^{-1};u}(\bm{z}),
\end{align}
where the kernel $K^{\rho^{+},\rho^{-}}_{q,t;u;\mcal{E}}(z,w):=K^{\rho^{+},\rho^{-};1}_{q,t;u;\mcal{E}}(z,w)$ is simplified so that
\begin{align}
	K^{\rho^{+},\rho^{-}}_{q,t;u;\mcal{E}}(z,w)=\frac{1}{z-t^{-1}w}\prod_{n>0}\exp\left(\frac{1-t^{-n}}{1-u^{n}}\frac{p_{n}(\rho^{+})}{n}z^{n}-\frac{1-t^{n}}{1-u^{n}}\frac{p_{n}(\rho^{-})}{n}z^{-n}\right).
\end{align}
\end{cor}
Here the measure part $\Delta_{q,t^{-1};u}(\bm{z})$ reduces to unity at $u\to 0$ so that the expectation value is written as multiple integral of a single determinant. Therefore, at $u=0$, the expectation value of the generating function of $\mcal{E}_{r}$, $r\in\mbb{N}$ admits an expression as a Fredholm determinant. If $u>0$, however, the integrand is modified by the measure $\Delta_{q,t^{-1};u}(\bm{z})$.

We also consider the shift-mixed version of a periodic Macdonald measure, which is a probability measure on $\mbb{Y}\times\mbb{Z}$ extended from a periodic Macdonald measure ($N$-step periodic Macdonald processes can be also extended to shift-mixed ones in an obvious manner, so we omit them in the present paper). Recall that, in the Schur case \cite{Borodin2007,BeteaBouttier2019}, several determinantal formulas are only available after extending a periodic Schur measure to a shift-mixed one.

Let $q,t,u\in (0,1)$. To define a shift-mixed periodic Macdonald measure, we also take an additional parameter $\zeta \in (0,1)$.
\begin{defn}
Let $\rho^{+},\rho^{-}:\Lambda\to\mbb{R}$ be Macdonald positive specializations.
Then, the corresponding shift-mixed periodic Macdonald measure is the probability measure $\mbb{P}^{\rho^{+},\rho^{-}}_{q,t;u,\zeta}$ on $\mbb{Y}\times \mbb{Z}$ defined by
\begin{equation}
	\mbb{P}^{\rho^{+},\rho^{-}}_{q,t;u,\zeta}(\lambda,n)\propto u^{n^{2}/2}\zeta^{n}\sum_{\mu\in\mbb{Y}}u^{|\lambda|}P_{\lambda/\mu}(\rho^{+};q,t)Q_{\lambda/\mu}(\rho^{-};q,t).
\end{equation}
The partition function is immediately given by
\begin{align}
	\Pi_{q,t;u,\zeta}(\rho^{+};\rho^{-}):=&\sum_{n\in\mbb{Z}}u^{n^{2}/2}\zeta^{n}\sum_{\lambda,\mu\in\mbb{Y}}u^{|\lambda|}P_{\lambda/\mu}(\rho^{+};q,t)Q_{\lambda/\mu}(\rho^{-};q,t) \\
	=&\vartheta_{3}(\zeta;u)\Pi_{q,t;u}(\rho^{+},\rho^{-}), \notag
\end{align}
where $\vartheta_{3}(\zeta;u)=\sum_{n\in\mbb{Z}}u^{n^{2}/2}\zeta^{n}$ is a Jacobi theta function.
\end{defn}

Accordingly, the observables $\mcal{E}_{r}$, $r\in\mbb{N}$ for periodic Macdonald processes are extended onto $\mbb{Y}\times \mbb{Z}$ as follows:
for $r\in\mbb{N}$, we write $\tilde{\mcal{E}}_{r}:\mbb{Y}\times \mbb{Z}\to\mbb{R}$ for the random variable defined by $\tilde{\mcal{E}}_{r}(\lambda,n)=t^{-rn}\mcal{E}_{r}(\lambda)$, $\lambda\in\mbb{Y}$, $n\in\mbb{Z}$.
The expectation values of these observables under a shift-mixed periodic Macdonald measure are computed in the following theorem that will be proved in Sect.\ \ref{app:shift-mixed}.
\begin{thm}
\label{thm:expectation_shift_mixed}
For each $r\in\mbb{N}$, we have
\begin{equation}
\label{eq:expectation_shift_mixed}
	\mbb{E}^{\rho^{+},\rho^{-}}_{q,t;u,\zeta}[\tilde{\mcal{E}}_{r}]=
	\frac{1}{r!}\int D^{r}\bm{z}\det\left(K^{\rho^{+},\rho^{-}}_{q,t;u;\mcal{E}}(z_{i},z_{j})\right)_{1\le i,j\le r}\frac{\theta_{3}(\zeta t^{-r};u)}{\theta_{3}(\zeta;u)}\Delta_{q,t^{-1};u}(\bm{z}).
\end{equation}
\end{thm}

In Sect.\ \ref{app:shift-mixed}, we will also see that at the Schur-limit $q\to t$, the integrand of the formula (\ref{eq:expectation_shift_mixed}) reduces to a single determinant.
\begin{prop}
\label{prop:shift-mixed_Schur_limit}
For each $r\in\mbb{N}$, we have
\begin{equation}
	\mbb{E}^{\rho^{+},\rho^{-}}_{t,t;u,\zeta}[\tilde{\mcal{E}}_{r}]=\frac{1}{r!}\int D^{r}\bm{z}\det\left(K^{\rho^{+},\rho^{-}}_{t;u,\zeta;\tilde{\mcal{E}}}(z_{i},z_{j})\right)_{1\le i,j\le r},
\end{equation}
where
\begin{equation}
	K^{\rho^{+},\rho^{-}}_{t;u,\zeta;\tilde{\mcal{E}}}(z,w)=K^{\rho^{+},\rho^{-}}_{t,t;u,\zeta,\mcal{E}}(z,w)\frac{\theta_{3}(\zeta t^{-1}w/z;u)}{\theta_{3}(\zeta;u)}\frac{(u;u)_{\infty}^{2}}{(tuz/w;u)_{\infty}(t^{-1}uw/z;u)_{\infty}}.
\end{equation}
\end{prop}

All the above presented results are understood in the formal sense; the integrands are formal infinite series.
On the other hand, for any reasonable applications, the functional denoted by $\int D^{r}\bm{z}$ has to be understood as a multiple contour integral by regarding the integrand as a function.
In fact, as we will see in Sect.\ \ref{sect:observables}, to a given set of specializations, one can associate the canonical way of choosing integral contours along which the integrand is convergent to grant a formula an analytical sense, and then go on to further analysis.
At the same time, we still think that writing formulas in the formal sense has an advantage that the formulas are independent of a specific choice of specializations and we can separate the problem of integrability from that of actual analysis.

This paper is organized as follows:
The following Sect.\ \ref{sect:preliminaries} is devoted to preliminaries, where we will review some notions of symmetric functions and Fock representations of a Heisenberg algebra.
We will also present the trace formula of a vertex operator in an extended sense and recall the free field realization of operators diagonalized by the Macdonald symmetric functions.
In Sect.\ \ref{sect:periodic_Macdonald}, we introduce a periodic Macdonald process. We apply the trace formula of a vertex operator in Sect.\ \ref{sect:preliminaries} to the computation of a partition function, and prove Proposition \ref{prop:Cauchy_identity}.
In Sect.\ \ref{eq:Plancherel_process}, after recalling a Plancherel specialization of the Macdonald symmetric functions, we define a periodic continuous process and show that its marginal law for a fixed sequence of times is a periodic Macdonald process.
In Sect.\ \ref{sect:observables}, we give a proof of Theorem \ref{thm:first_series} and study some other series of observables.
We also study the analytic interpretation of several formulas there.
In Sect.\ \ref{app:shift-mixed}, we consider the shift-mixed version of a periodic Macdonald measure and prove Theorem \ref{thm:expectation_shift_mixed} and its Schur-limit, Proposition \ref{prop:shift-mixed_Schur_limit}. We also show that the Macdonald operators at the Schur-limit are written in terms of free fermions, which gives the origin of the well-known determinantal structure of Macdonald processes.
We have two appendices which deal with combinatorial aspects of periodic Macdonald processes.
In Appendix \ref{app:MacMahon_cylindric_partitions}, we study a generalization of the MacMahon formula for cylindric partitions.
In Appendix \ref{app:another_trace_topolgical_vertex}, we give the trace of Macdonald refined topological vertices proposed in \cite{FodaWu2017} as an attempt to generalize the results in \cite{BryanKoolYoung2018}.

\subsection*{Acknowledgements}
This work stems from a suggestion by Alexey Bufetov at School and Workshop on Random Matrix Theory and Point Processes held in International Center for Theoretical Physics, Trieste after the author's talk there about the free field approach to Macdonald processes. The author is grateful to Alexey Bufetov and the organizers of that workshop, especially to Alexander Bufetov for inviting the author. He also thanks Makoto Katori, Tomohiro Sasamoto, Takashi Imamura, Matteo Mucciconi and Ryosuke Sato for joining seminars and discussions.
The author appreciates the anonymous referees' useful comments to improve the manuscript as well.
This work was supported by the Grant-in-Aid for JSPS Fellows (No.\,19J01279).

\section{Preliminaries}
\label{sect:preliminaries}
\subsection{Symmetric functions}
Here, we recall the basic notion of symmetric functions. A detail can be found in \cite{Macdonald1995}.
Let us begin with recalling that a partition is a non-decreasing sequence of non-negative integers $\lambda=(\lambda_{1},\lambda_{2},\dots)$ such that its weight is finite: $|\lambda|=\sum_{i\ge 1}\lambda_{i}<\infty$. The collection of partitions is denoted by $\mbb{Y}$.

Let $\mbb{F}=\mbb{Q}(q,t)$ be the field of rational functions of $q$ and $t$.
The ring of symmetric polynomials of $n$ variables is denoted by $\Lambda^{(n)}=\mbb{F}[x_{1},\dots, x_{n}]^{\mfrak{S}_{n}}$ and the ring of symmetric functions is defined by $\Lambda:=\varprojlim_{n}\Lambda^{(n)}$ in the category of graded rings.
In case we specify the variable $X=(x_{1},x_{2},\dots)$ of symmetric functions, we write $\Lambda_{X}$ for the corresponding ring of symmetric functions.
For a symmetric function $F\in \Lambda$, its $n$-variable reduction, i.e. the image of the canonical projection $\Lambda\to\Lambda^{(n)}$ is denoted by $F^{(n)}$.

We recall some series of symmetric functions. For $r\in\mbb{N}$, the $r$-th power-sum symmetric function is $p_{r}(X)=\sum_{i\ge 1}x_{i}^{r}$ and the $r$-th elementary symmetric function is $e_{r}(X)=\sum_{i_{1}<\cdots <i_{r}}x_{i_{1}}\cdots x_{i_{r}}$.
For a partition $\lambda=(\lambda_{1},\lambda_{2},\dots)$ we set $p_{\lambda}=p_{\lambda_{1}}\cdots p_{\lambda_{\ell(\lambda)}}$ and $e_{\lambda}=e_{\lambda_{1}}\cdots e_{\lambda_{\ell (\lambda)}}$. Then the collections $\set{p_{\lambda}:\lambda \in\mbb{Y}}$ and $\set{e_{\lambda}:\lambda\in\mbb{Y}}$ are both bases of $\Lambda$.
Take a partition $\lambda=(\lambda_{1},\dots, \lambda_{n})$ of length less than or equal to $n$. Then the corresponding monomial symmetric polynomial of $n$-variable is defined by
\begin{equation}
	m^{(n)}_{\lambda}(x_{1},\dots, x_{n})=\sum_{\sigma\in \mfrak{S}_{n}: \mrm{distinct}}x_{1}^{\lambda_{\sigma (1)}}\cdots x_{n}^{\lambda_{\sigma (n)}},
\end{equation}
where the sum runs over distinct permutations $\sigma$ of $(\lambda_{1},\dots,\lambda_{n})$.
Then the sequence $\cdots\to m_{\lambda}^{(\ell (\lambda)+2)}\to m_{\lambda}^{(\ell (\lambda)+1)}\to m_{\lambda}^{(\ell (\lambda))}$ defines a unique symmetric function $m_{\lambda}(X)\in\Lambda$ called the monomial symmetric function corresponding to $\lambda$.
The collection $\set{m_{\lambda}:\lambda\in\mbb{Y}}$ is known to form a basis of $\Lambda$ as well.

To the aim of introducing the Macdonald symmetric functions, let us define the bilinear form $\braket{\cdot,\cdot}_{q,t}:\Lambda\times\Lambda\to\mbb{F}$ by
\begin{equation}
	\braket{p_{\lambda},p_{\mu}}_{q,t}=z_{\lambda}(q,t)\delta_{\lambda,\mu}, \quad \lambda,\mu\in\mbb{Y},
\end{equation}
where
\begin{equation}
	z_{\lambda}(q,t):=\prod_{i\ge 1}m_{i}(\lambda)! i^{m_{i}(\lambda)}\prod_{i=1}^{\ell(\lambda)}\frac{1-q^{\lambda_{i}}}{1-t^{\lambda_{i}}}, \quad \lambda\in \mbb{Y}.
\end{equation}
Here, we write the multiplicity of a number $i$ in a partition $\lambda$ as $m_{i}(\lambda)$ so that a partition $\lambda=(\lambda_{1},\lambda_{2},\dots)$ is equivalently expressed as $\lambda=(1^{m_{1}(\lambda)}2^{m_{2}(\lambda)}\dots)$.
Recall that the dominance order on $\mbb{Y}$ is defined so that $\lambda\ge\mu$ if
\begin{equation}
	\lambda_{1}+\cdots+\lambda_{k}\ge \mu_{1}+\cdots +\mu_{k}
\end{equation}
holds for all $k=1,2,\dots$.
The Macdonald symmetric functions $\set{P_{\lambda}(q,t):\lambda\in\mbb{Y}}$ form a unique basis of $\Lambda$ specified by the following properties:
\begin{align}
	&P_{\lambda}(q,t)=m_{\lambda}+\sum_{\mu;\mu<\lambda}c_{\lambda\mu}m_{\mu},\ \ c_{\lambda\mu}\in\mbb{F}, \\
	&\braket{P_{\lambda}(q,t),P_{\lambda}(q,t)}_{q,t}=0,\ \ \lambda\neq \mu.
\end{align}
We set $Q_{\lambda}(q,t)=\frac{1}{\braket{P_{\lambda}(q,t),P_{\lambda}(q,t)}_{q,t}}P_{\lambda}(q,t)$ so that $\braket{P_{\lambda}(q,t),Q_{\mu}(q,t)}_{q,t}=\delta_{\lambda,\mu}$, $\lambda,\mu\in\mbb{Y}$.

The Macdonald symmetric functions are also characterized as simultaneous eigenfunctions of a family of commuting operators.
For a fixed $n<\infty$, let us introduce the Macdonald difference operators $D^{(n)}_{r}$, $r=1,\dots, N$ on $\Lambda^{(n)}$ by
\begin{equation}
	D^{(n)}_{r}=t^{\binom{r}{2}}\sum_{I\subset [1,n], |I|=r}\prod_{i\in I, j\not\in I}\frac{tx_{i}-x_{j}}{x_{i}-x_{j}}\prod_{i\in I}T_{q,x_{i}},
\end{equation}
where $T_{q,x_{i}}$ is the $q$-shift operator defined by $(T_{q,x_{i}}f)(\dots, x_{i},\dots)=f(\dots, qx_{i},\dots)$.
Then, a Macdonald symmetric polynomial $P^{(n)}_{\lambda}(q,t)$ is an eigenpolynomial such that
\begin{equation}
	D^{(n)}_{r}P^{(n)}_{\lambda}(q,t)=e^{(n)}_{r}(q^{\lambda_{1}}t^{n-1},\dots, q^{\lambda_{n}})P^{(n)}_{\lambda}(q,t),\ \ r=1,\dots, n.
\end{equation}
Obviously, the Macdonald difference operators do not extend to operators on $\Lambda$ since their eigenvalues explicitly depend on the number of variables. Instead, we define renormalized operators
\begin{equation}
	E^{(n)}_{r}:=\sum_{k=0}^{r}\frac{t^{-nr-\binom{r-k+1}{2}}}{(t^{-1};t^{-1})_{r-k}}D^{(n)}_{k},\ \ r=1,\dots, n,
\end{equation}
with $D^{(n)}_{0}=1$. Then, the projective limit $E_{r}=\varprojlim_{n}E^{(n)}_{r}$ exists for all $r\in\mbb{N}$ and is diagonalized by the Macdonald symmetric functions so that
\begin{equation}
	E_{r}P_{\lambda}(q,t)=e_{r}(q^{\lambda}t^{-\rho})P_{\lambda}(q,t),\ \ \lambda\in\mbb{Y}.
\end{equation}
Here, $q^{\lambda}t^{-\rho}\colon \Lambda\to \mbb{F}$ is the specialization defined by Eq.~(\ref{eq:specialization_lambda_rho}).

Let us describe some Pieri rules for the Macdonald symmetric functions.
For a partition $\lambda$ and a box $s=(i,j)\in \lambda$, we write $a_{\lambda}(s)=\lambda_{i}-j$ and $l_{\lambda}(s)=\pr{\lambda}_{j}-i$ for the arm length and the leg length, respectively. Using these notions, we set
\begin{equation}
	b_{\lambda}(s;q,t)=\frac{1-q^{a_{\lambda}(s)}t^{l_{\lambda}(s)+1}}{1-q^{a_{\lambda}(s)+1}t^{l_{\lambda}(s)}}.
\end{equation}
For a skew partition $\lambda/\mu$, we express the union of rows (resp. columns) containing boxes in $\lambda/\mu$ as $R_{\lambda/\mu}$ (resp. $C_{\lambda/\mu}$).
Assume that $\mu\prec \lambda$ and set
\begin{align}
	\psi_{\lambda/\mu}(q,t)&:=\prod_{s\in R_{\lambda/\mu}-C_{\lambda/\mu}}\frac{b_{\mu}(s;q,t)}{b_{\lambda}(s;q,t)}, & 
	\varphi_{\lambda/\mu}(q,t)&:=\prod_{s\in C_{\lambda/\mu}}\frac{b_{\lambda}(s;q,t)}{b_{\mu}(s;q,t)}.
\end{align}
For $r\in\mbb{N}$, we write $g_{r}(q,t)=Q_{(r)}(q,t)$. Then we have
\begin{align}
	P_{\mu}(q,t)g_{r}(q,t)&=\sum_{\substack{\lambda; \mu\prec \lambda \\|\;\lambda|-|\mu|=r}}\varphi_{\lambda/\mu}(q,t)P_{\lambda}(q,t), \\
	Q_{\mu}(q,t)g_{r}(q,t)&=\sum_{\substack{\lambda; \mu\prec \lambda \\|\;\lambda|-|\mu|=r}}\psi_{\lambda/\mu}(q,t)Q_{\lambda}(q,t).
\end{align}

Though we have been regarding $q$ and $t$ as indeterminates, they have to be thought of as real parameters in application to probability theory.
Let $\Lambda_{\mbb{R}}=\varprojlim_{n}\Lambda^{(n)}_{\mbb{R}}$, $\Lambda^{(n)}_{\mbb{R}}=\mbb{R}[x_{1},\dots, x_{n}]^{\mfrak{S}_{n}}$ be the ring of symmetric functions over $\mbb{R}$.
For $q,t\in\mbb{R}$ such that $|q|,|t|<1$, we write, by abuse of notation, $P_{\lambda}(q,t)\in\Lambda_{\mbb{R}}$ for the image of $P_{\lambda}(q,t)\in\Lambda$ under the specification map $\bigoplus_{\lambda\in\mbb{Y}}\mbb{Q}[q,t]P_{\lambda}(q,t)\to\Lambda_{\mbb{R}}$.
A specialization $\theta:\Lambda_{\mbb{R}}\to\mbb{R}$ is said to be $(q,t)$-Macdonald positive if $P_{\lambda}(\theta; q,t)\ge 0$ for all $\lambda\in\mbb{Y}$.
Note that, for a specialization $\theta$, its evaluation at $F\in\Lambda_{\mbb{R}}$ is often written as $F(\theta)$ instead of $\theta (F)$.
The following theorem is due to \cite{Matveev2019}.
\begin{thm}
For fixed $q,t\in\mbb{R}$ such that $|q|,|t|<1$, a specialization $\theta:\Lambda_{\mbb{R}}\to\mbb{R}$ is $(q,t)$-Macdonald positive if and only if there exists $(\alpha,\beta)=(\alpha_{1}\ge \alpha_{2}\ge\cdots\ge 0 , \beta_{1}\ge \beta_{2}\ge\cdots\ge 0)$ and $\gamma\ge 0$ such that $\sum_{i\ge 1}(\alpha_{i}+\beta_{i})<\infty$ and
\begin{align}
	p_{1}(\theta)&=\sum_{i=1}^{\infty}\alpha_{i}+\frac{1-q}{1-t}\left(\sum_{i=1}^{\infty}\beta_{i}+\gamma\right), \\
	p_{k}(\theta)&=\sum_{i=1}^{\infty}\alpha_{i}^{k}+(-1)^{k-1}\frac{1-q^{k}}{1-t^{k}}\sum_{i=1}^{\infty}\beta_{i}^{k},\ \ k\ge 2.
\end{align}
\end{thm}

Let $X$ and $Y$ be two sets of infinite variables. Then, they are combined to be a single set of infinite variables $(X,Y)$ and the Macdonald symmetric function $P_{\lambda}(X,Y)$ of $(X,Y)$ corresponding to a partition $\lambda$ makes sense. The Macdonald symmetric functions corresponding to skew-partitions are defined as the coefficients in the expansion of $P_{\lambda}(X,Y)$ with respect to the Macdonald symmetric functions of $Y$, i.e., $P_{\lambda/\mu}(X)$ is defined by
$P_{\lambda/\mu}(X,Y)=\sum_{\mu\in\mbb{Y}}P_{\lambda/\mu}(X)P_{\mu}(Y)$.
Similarly, we define $Q_{\lambda/\mu}(X)$ by $Q_{\lambda}(X,Y)=\sum_{\mu\in\mbb{Y}}Q_{\lambda/\mu}(X)Q_{\mu}(Y)$.

\subsection{Fock representations of a Heisenberg algebra}
\label{subsect:Fock_representation}
The Heisenberg algebra we work with is an associative algebra $U$ over $\mbb{F}$ generated by $a_{n}$, $n\in\mbb{Z}\backslash\set{0}$ subject to relations
\begin{equation}
	[a_{m},a_{n}]=m\frac{1-q^{|m|}}{1-t^{|m|}}\delta_{m+n,0},\ \ m,n\in\mbb{Z}\backslash\set{0}.
\end{equation}
Note that the Heisenberg algebra admits Poincar\'{e}--Birkhoff--Witt decomposition $U=U_{-}\otimes_{\mbb{F}}U_{+}$, where $U_{\pm}$ is the subalgebra generated by $a_{\pm n}$, $n>0$.

Let $\mbb{F}\ket{0}$ be the one dimensional representation of $U_{+}$ defined by $U_{+}\ket{0}=0$. Then, the Fock representation $\mcal{F}$ is defined by means of induction so that $\mcal{F}=U\otimes_{U_{+}}\mbb{F}\ket{0}\simeq U_{-}\otimes_{\mbb{F}}\mbb{F}\ket{0}$.
Obviously, it admits a basis labeled by partitions. For a partition $\lambda=(\lambda_{1},\dots,\lambda_{\ell(\lambda)})\in\mbb{Y}$, we set $\ket{\lambda}:=a_{-\lambda_{1}}\cdots a_{-\lambda_{\ell(\lambda)}}\ket{0}$. Then, the collection $\{\ket{\lambda}:\lambda\in\mbb{Y}\}$ forms a basis of $\mcal{F}$.
The dual Fock representation is a right representation of $U$ defined analogously. Let $\mbb{F}\bra{0}$ be the one dimensional representation of $U_{-}$ such that $\bra{0}U_{-}=0$ and define a right representation by $\mcal{F}^{\dagger}=\mbb{F}\bra{0}\otimes_{U_{-}}U\simeq \mbb{F}\bra{0}\otimes_{\mbb{F}}U_{+}$.
For a partition $\lambda=(\lambda_{1},\dots,\lambda_{\ell(\lambda)})\in\mbb{Y}$, we set $\bra{\lambda}=\bra{0}a_{\lambda_{1}}\cdots a_{\lambda_{\ell(\lambda)}}$. Then the collection $\{\bra{\lambda}:\lambda\in\mbb{Y}\}$ forms a basis of $\mcal{F}^{\dagger}$.

We define the bilinear paring $\braket{\cdot|\cdot}:\mcal{F}^{\dagger}\times \mcal{F}\to\mbb{F}$ by two properties, $\braket{0|0}=1$ and $\braket{u|a_{n}\cdot|v}=\braket{u|\cdot a_{n}|v}$ for all $\bra{u}\in\mcal{F}^{\dagger}$, $\ket{v}\in\mcal{F}$ and $n\in\mbb{Z}\backslash\set{0}$.
Then the assignments $\iota:\mcal{F}\to\Lambda$ and $\iota^{\dagger}:\mcal{F}^{\dagger}\to\Lambda$ defined by $\iota(\ket{\lambda})=\iota^{\dagger}(\bra{\lambda})=p_{\lambda}$, $\lambda\in\mbb{Y}$ are compatible with the bilinear parings so that $\braket{\lambda|\mu}=\braket{p_{\lambda},p_{\mu}}_{q,t}$, $\lambda,\mu\in\mbb{Y}$. Therefore, the Fock spaces are identified with the space of symmetric functions $\Lambda$ equipped with the bilinear pairing $\braket{\cdot,\cdot}_{q,t}$.

A relevant class of operators in this paper is that of vertex operators (in an extended sense). For commuting symbols $\gamma_{n}$, $n\in\mbb{Z}\backslash\set{0}$, we write
\begin{equation}
	V(\bm{\gamma})=\exp\left(\sum_{n>0}\frac{\gamma_{-n}}{n}a_{-n}\right)\exp\left(\sum_{n>0}\frac{\gamma_{n}}{n}a_{n}\right),	
\end{equation}
where $a_{n}\in\mrm{End}(\mcal{F})$ is the action of the corresponding Heisenberg generator on $\mcal{F}$.
In this paper, we will adopt following two manners to understand a vertex operator.
\begin{enumerate}
\item 	There is a $\mbb{Z}$-graded commutative algebra $R=\bigoplus_{n\in\mbb{Z}}R_{n}$ over $\mbb{F}$ such that $\gamma_{n}\in R_{n}$, $n\in\mbb{Z}\backslash\set{0}$. We define the completed tensor product $\mrm{End}(\mcal{F})\hat{\otimes}R:=\prod_{n\in\mbb{Z}}\mrm{End}(\mcal{F})\otimes_{\mbb{F}}R_{n}$. Then, the vertex operator is understood as $V(\bm{\gamma})\in\mrm{End}(\mcal{F})\hat{\otimes}R$. A particular example of $R$ is the ring of Laurent polynomials $\mbb{F}[z,z^{-1}]$.
\item 	There are $\mbb{N}$-graded commutative algebras $R^{i}=\bigoplus_{n\in\mbb{N}}R^{i}_{n}$, $i=1,2$ over $\mbb{F}$ such that $\gamma_{n}\in R^{1}_{n}$, $\gamma_{-n}\in R^{2}_{n}$, $n>0$. The tensor product of them also admits an $\mbb{N}$-gradation $R^{1}\otimes R^{2}=\bigoplus_{n\in\mbb{N}}(R^{1}\otimes R^{2})_{n}$. Then we can understand $V(\bm{\gamma})\in\mrm{End}(\mcal{F})\hat{\otimes}(R^{1}\otimes R^{2})$. An example is the case that $R^{1}=R^{2}=\Lambda$.
\end{enumerate}

The isomorphisms $\iota$ and $\iota^{\dagger}$ are realized as computation of matrix elements of vertex operators. Let us introduce
\begin{equation}
	\Gamma (X)_{\pm}=\exp\left(\sum_{n>0}\frac{1-t^{n}}{1-q^{n}}\frac{p_{n}(X)}{n}a_{\pm n}\right).
\end{equation}
Notice that these are vertex operators specified by $\gamma_{\pm n}=\frac{1-t^{n}}{1-q^{n}}p_{n}(X)\in \Lambda_{X}$, $\gamma_{\mp n}=0$, $n>0$.
We will rely on extensive use of the following fact.
\begin{prop}
For any $\ket{v}\in\mcal{F}$ or $\bra{v}\in\mcal{F}^{\dagger}$, its image under $\iota$ or $\iota^{\dagger}$ is realized as
\begin{align}
	\iota(\ket{v})&=\braket{0|\Gamma (X)_{+}|v}\in\Lambda_{X}, &	\iota^{\dagger}(\bra{v})&=\braket{v|\Gamma (X)_{-}|0}\in\Lambda_{X}.
\end{align}	
\end{prop}

For a symmetric function $F\in\Lambda$, we write $\ket{F}:=\iota^{-1}(F)\in\mcal{F}$ and $\bra{F}:=(\iota^{\dagger})^{-1}(F)\in\mcal{F}^{\dagger}$.
It immediately follows that
\begin{align}
	P_{\lambda/\mu}(X)&=\braket{Q_{\mu}|\Gamma (X)_{+}|P_{\lambda}}=\braket{P_{\lambda}|\Gamma (X)_{-}|Q_{\mu}}, \\
	Q_{\lambda/\mu}(X)&=\braket{P_{\mu}|\Gamma (X)_{+}|Q_{\lambda}}=\braket{Q_{\lambda}|\Gamma (X)_{-}|P_{\mu}},
\end{align}
for any skew-partition $\lambda/\mu$.

In this paper, the trace of a vertex operator plays a significant role.
Let $D\in\mrm{End}(\mcal{F})$ be the degree operator defined by $D\ket{\lambda}=|\lambda|\ket{\lambda}$, $\lambda\in\mbb{Y}$.
The following formula was essentially found in e.g. \cite{DongMason2000,ChengWang2007}.
\begin{prop}
\label{prop:trace_formula}
Let $\gamma_{n}$, $n\in\mbb{Z}\backslash\set{0}$ be commuting symbols and let $u$ be yet another formal variable.
Then we have
\begin{equation}
\label{eq:trace_vertex_operator}
	\mrm{Tr}_{\mcal{F}}\left(u^{D}V(\bm{\gamma})\right)=\frac{1}{(u;u)_{\infty}}\exp\left(\sum_{n>0}\frac{1-q^{n}}{1-t^{n}}\frac{u^{n}}{1-u^{n}}\frac{\gamma_{-n}\gamma_{n}}{n}\right).	
\end{equation}
\end{prop}

\begin{rem}
The both sides of (\ref{eq:trace_vertex_operator}) are understood as follows.
\begin{enumerate}
\item 	If $\gamma_{n}\in R_{n}$, $n\in\mbb{Z}\backslash\set{0}$ with a $\mbb{Z}$-graded commutative algebra $R$, the trace lies in $\overline{R}[[u]]$, where $\overline{R}=\prod_{n\in\mbb{Z}}R_{n}$.
\item 	If $\gamma_{n}\in R_{n}^{1}$ and $\gamma_{-n}\in R^{2}_{n}$, $n>0$ with $\mbb{N}$-graded commutative algebras $R^{1}$, $R^{2}$, the trace lies in $\overline{R^{1}\otimes R^{2}}[[u]]$.	
\end{enumerate}
\end{rem}

\begin{proof}
For a partition $\lambda\in\mbb{Y}$, we write $\pi_{\lambda}:\mcal{F}\to\mbb{F}\ket{\lambda}$ for the projection with respect to the basis $\{\ket{\lambda}:\lambda\in\mbb{Y}\}$.
Then, we can see that
\begin{align}
	\pi_{\lambda}V(\bm{\gamma})\ket{\lambda}=\prod_{n>0}\left(\sum_{k=0}^{m_{n}(\lambda)}\frac{1}{(k!)^{2}}\left(\frac{\gamma_{-n}\gamma_{n}}{n^{2}}\right)^{k}(a_{-n})^{k}(a_{n})^{k}\right)\ket{\lambda}.
\end{align}
Here, note that
\begin{align}
	(a_{-n})^{k}(a_{n})^{k}(a_{-n})^{m_{n}(\lambda)}\ket{0}=k!\binom{m_{n}(\lambda)}{k}\left(n\frac{1-q^{n}}{1-t^{n}}\right)^{k}(a_{-n})^{m_{n}(\lambda)}\ket{0}
\end{align}
to find
\begin{align}
	\pi_{\lambda}V(\bm{\gamma})\ket{\lambda}=\prod_{n>0}\left(\sum_{k=0}^{m_{n}(\lambda)}\frac{1}{k!}\binom{m_{n}(\lambda)}{k}\left(\frac{1-q^{n}}{1-t^{n}}\frac{\gamma_{-n}\gamma_{n}}{n}\right)^{k}\right)\ket{\lambda}.
\end{align}
Since $|\lambda|=\sum_{n>0}m_{n}(\lambda)n$, the trace becomes
\begin{align}
	\mrm{Tr}_{\mcal{F}}\left(u^{D}V(\bm{\gamma})\right)=\prod_{n>0}\left(\sum_{k=0}^{\infty} \frac{1}{k!}\left(\frac{1-q^{n}}{1-t^{n}}\frac{\gamma_{-n}\gamma_{n}}{n}\right)^{k}\sum_{m=k}^{\infty}\binom{m}{k}u^{nm}\right).
\end{align}
Notice
\begin{equation}
	\sum_{m=k}^{\infty}\binom{m}{k}x^{m}=\frac{x^{k}}{(1-x)^{k+1}},	
\end{equation}
which implies
\begin{align}
	\mrm{Tr}_{\mcal{F}}\left(u^{D}V(\bm{\gamma})\right)&=\frac{1}{(u;u)_{\infty}}\prod_{n>0}\left(\sum_{k=0}^{\infty} \frac{1}{k!}\left(\frac{1-q^{n}}{1-t^{n}}\frac{u^{n}}{1-u^{n}}\frac{\gamma_{-n}\gamma_{n}}{n}\right)^{k}\right) \\
	&=\frac{1}{(u;u)_{\infty}}\exp\left(\sum_{n>0} \frac{1-q^{n}}{1-t^{n}}\frac{u^{n}}{1-u^{n}}\frac{\gamma_{-n}\gamma_{n}}{n}\right) \notag
\end{align}
completing the proof.
\end{proof}

Associated with $\bm{\gamma}^{i}=(\gamma^{i}_{n}:n\in\mbb{Z}\backslash\set{0})$, $i=1,\dots, N$, the normally ordered product of corresponding vertex operators is defined by
\begin{equation}
	\no{V(\bm{\gamma}^{1})\cdots V(\bm{\gamma}^{N})}=\exp\left(\sum_{n>0}\frac{\sum_{i=1}^{N}\gamma^{i}_{-n}}{n}a_{-n}\right)\exp\left(\sum_{n>0}\frac{\sum_{i=1}^{N}\gamma^{i}_{n}}{n}a_{n}\right).	
\end{equation}
That is, we put positive modes of the Heisenberg algebra on the right and negative modes on the left.
Besides, note that, in the normally ordered product, vertex operators are commutative.

The following formula will be extensively used under the name of the operator product expansion (OPE):
\begin{prop}
\label{prop:OPE}
Let $\bm{\gamma}=(\gamma_{n}:n\in\mbb{Z}\backslash\set{0})$ and $\pr{\bm{\gamma}}=(\pr{\gamma}_{n}:n\in\mbb{Z}\backslash\set{0})$ be sequences of commuting symbols. Then the corresponding vertex operators admit the following formula:
\begin{equation}
	V(\bm{\gamma})V(\pr{\bm{\gamma}})=\exp\left(\sum_{n>0}\frac{1-q^{n}}{1-t^{n}}\frac{\gamma_{n}\pr{\gamma}_{-n}}{n}\right)\no{V(\bm{\gamma})V(\pr{\bm{\gamma}})}.
\end{equation}
\end{prop}
\begin{proof}
It follows from the Baker--Campbell--Hausdorff formula. Note that, since the commutator of two Heisenberg generators is central, higher commutators do not appear in the exponential.
\end{proof}

\subsection{Free field realization of operators}
\label{subsect:free_field_operators}
Since we have identified the Fock space $\mcal{F}$ with the space of symmetric functions, we can also identify an operator $T\in \mrm{End}(\Lambda)$ with one $\what{T}:=\iota^{-1}\circ T \circ \iota\in\mrm{End}(\mcal{F})$, the free field realization of $T$.
We introduced the family of renormalized Macdonald operators $E_{r}$, $r\in\mbb{N}$.
To describe their free field realization, we introduce the vertex operator
\begin{equation}
	\eta (z)=V(\bm{\gamma})=\exp\left(\sum_{n>0}\frac{1-t^{-n}}{n}a_{-n}z^{n}\right)\exp\left(-\sum_{n>0}\frac{1-t^{n}}{n}a_{n}z^{-n}\right)	
\end{equation}
associated with $\gamma_{n}=-\sgn (n)(1-t^{n})z^{-n}\in\mbb{F}[z,z^{-1}]$, $n\in\mbb{Z}\backslash\set{0}$.

\begin{prop}[\cite{Shiraishi2006,FeiginHashizumeHoshinoShiraishiYanagida2009,SK2019}]
\label{prop:free_field_Macdonald_operator}
For each $r\in\mbb{N}$, the free field realization of the $r$-th renormalized Macdonald operator is given by
\begin{equation}
		\what{E}_{r}=\frac{t^{-r}}{r!}\int D^{r}\bm{z}\det\left(\frac{1}{z_{i}-t^{-1}z_{j}}\right)_{1\le i,j\le r}\no{\eta (z_{1})\cdots \eta (z_{r})}.
\end{equation}
Here, a rational function of the form $\frac{1}{x-\gamma y}:=\sum_{k=0}^{\infty}x^{-k-1}(\gamma y)^{k}$, $\gamma\in\mbb{F}$ is always expanded in $\mbb{F}[x,x^{-1}][[y]]$ (Therefore, $\frac{1}{x-\gamma y}\neq -\frac{1}{\gamma y-x}$).
\end{prop}

We stress here that the operators $\what{E}_{r}$, $r\in\mbb{N}$ act on the Fock space $\mcal{F}$, which is isomorphic to the space of symmetric functions $\Lambda$. Therefore, they no longer make sense as difference operators, but they are projective limits of difference operators.

Owing to the property $P_{\lambda}(q,t)=P_{\lambda}(q^{-1},t^{-1})$ of a Macdonald symmetric function, the renormalized Macdonald operator with inverted parameters $\pr{E}_{r}:=E_{r}(q^{-1},t^{-1})$ is also diagonalized by the Macdonald symmetric functions.
Let us introduce another vertex operator
\begin{equation}
	\xi (z)=V(\bm{\gamma})=\exp\left(-\sum_{n>0}\frac{1-t^{-n}}{n}(t/q)^{n/2}a_{-n}z^{n}\right) \exp\left(\sum_{n>0}\frac{1-t^{n}}{n}(t/q)^{n/2}a_{n}z^{-n}\right)	
\end{equation}
associated with $\gamma_{n}=\sgn (n)(1-t^{n})(t/q)^{|n|/2}z^{-n}\in \mbb{Q}(q^{1/2},t^{1/2})[z,z^{-1}]$, $n\in\mbb{Z}\backslash\set{0}$.
\begin{prop}[\cite{Shiraishi2006,FeiginHashizumeHoshinoShiraishiYanagida2009,SK2019}]
For each $r\in\mbb{N}$, we have
\begin{equation}
	\what{E}^{\prime}_{r}=\frac{t^{r}}{r!}\int D^{r}\bm{z}\det\left(\frac{1}{z_{i}-tz_{j}}\right)_{1\le i,j\le r}\no{\xi (z_{1})\cdots \xi (z_{r})}.	
\end{equation}
\end{prop}
Then, by definition, these operators are diagonalized by the Macdonald basis so that
\begin{equation}
	\what{E}^{\prime}_{r}\ket{P_{\lambda}(q,t)}=e_{r}(q^{-\lambda}t^{\rho})\ket{P_{\lambda}(q,t)},\quad r\in\mbb{N},\quad \lambda\in\mbb{Y}.	
\end{equation}

Let us see still other two series of operators diagonalized by the Macdonald basis following \cite{Shiraishi2006,FeiginHashizumeHoshinoShiraishiYanagida2009,SK2019}.
For $r\in\mbb{N}$, we introduce the following operator
\begin{equation}
	\what{G}_{r}:=\frac{(-1)^{r}}{r!}\int D^{r}\bm{z}\det\left(\frac{1}{z_{i}-qz_{j}}\right)_{1\le i,j\le r}\no{\eta (z_{1})\cdots \eta (z_{r})}.
\end{equation}
Then, we have
\begin{equation}
	\what{G}_{r}\ket{P_{\lambda}(q,t)}=g_{r}(q^{\lambda}t^{-\rho};q,t)\ket{P_{\lambda}(q,t)},\quad r\in\mbb{N},\quad \lambda\in\mbb{Y}.	
\end{equation}
The other series corresponds to $\what{G}_{r}$, $r\in\mbb{N}$ with inverted parameters:
\begin{equation}
	\what{G}^{\prime}_{r}=\frac{(-1)^{r}}{r!} \int D^{r}\bm{z}\det\left(\frac{1}{z_{i}-q^{-1}z_{j}}\right)_{1\le i,j\le r}\no{\xi (z_{1})\cdots \xi (z_{r})},\quad r\in\mbb{N}.
\end{equation}
Then they are diagonalized so that
\begin{equation}
	\what{G}^{\prime}_{r}\ket{P_{\lambda}(q,t)}=g_{r}(q^{-\lambda}t^{\rho};q^{-1},t^{-1})\ket{P_{\lambda}(q,t)},\quad r\in\mbb{N},\quad \lambda\in\mbb{Y}.	
\end{equation}
Here, note that the eigenvalue is the same as $g_{r}(q^{-\lambda}t^{\rho};q^{-1},t^{-1})=g_{r}(q^{-\lambda+1}t^{\rho -1};q,t)$.

\section{Periodic Macdonald processes}
\label{sect:periodic_Macdonald}
The aim of this section is to prove Proposition~\ref{prop:Cauchy_identity}.
Let us define the universal version of the weight (\ref{eq:periodic_weight}).
Let $N\in\mbb{N}$.
We write $\bm{X}=(X^{0},\dots, X^{N-1})$ and $\bm{Y}=(Y^{1},\dots, Y^{N})$ for $N$-tuples of infinite set of variables.
For $\bm{\lambda}=(\lambda^{1},\dots,\lambda^{N}), \bm{\mu}=(\mu^{1},\dots, \mu^{N})\in\mbb{Y}^{N}$, we set
\begin{equation}
	W^{\bm{X},\bm{Y}}_{q,t;u}(\bm{\lambda},\bm{\mu})=u^{|\mu^{N}|}\prod_{i=1}^{N}Q_{\lambda^{i}/\mu^{i}}(Y^{i};q,t)P_{\lambda^{i+1}/\mu^{i}}(X^{i};q,t) \in \left(\bigotimes_{i=0}^{N-1}\Lambda_{X^{i}}\otimes \Lambda_{Y^{i+1}}\right)[u].
\end{equation}
When we fix parameters $q,t\in\mbb{R}$ such that $|q|,|t|<1$, adopt specification at these parameters and take $N$-tuples of $(q,t)$-Macdonald positive specializations $\bm{\rho}^{+}=(\rho^{+}_{0},\dots, \rho^{+}_{N-1})$, $\bm{\rho}^{-}=(\rho^{-}_{1},\dots, \rho^{-}_{N})$, the weight (\ref{eq:periodic_weight}) is the image of $W^{\bm{X},\bm{Y}}_{q,t;u}(\bm{\lambda},\bm{\mu})$ under $\bigotimes_{i=0}^{N-1}\rho^{+}_{i}\otimes\rho^{-}_{i+1}$, where each specialization acts as $\rho^{+}_{i}:\Lambda_{\mbb{R},X^{i}}\to\mbb{R}$, $\rho^{-}_{i}:\Lambda_{\mbb{R},Y^{i}}\to\mbb{R}$.
Note that, at the limit $u\to 0$, the weight $W^{\bm{X},\bm{Y}}_{q,t;u}(\bm{\lambda},\bm{\mu})$ vanishes unless $\mu^{N}=\emptyset$ recovering the corresponding $N$-step Macdonald process since
\begin{equation}
	u^{|\mu|}\to
	\begin{cases}
	1, & \mu=\emptyset, \\
	0, & \mbox{otherwise}
	\end{cases}
\end{equation}
as $u\to 0$.

\begin{proof}[Proof of Proposition \ref{prop:Cauchy_identity}]
Our goal is to show
\begin{equation}
	\Pi_{q,t;u}(\bm{X};\bm{Y}):=\sum_{\bm{\lambda},\bm{\mu}\in\mbb{Y}}W^{\bm{X},\bm{Y}}_{q,t;u}(\bm{\lambda},\bm{\mu})
	=\frac{1}{(u;u)_{\infty}}\frac{\prod_{i=0}^{N-1}\prod_{j=1}^{N}\tilde{\Pi}_{q,t;u}(X^{i};Y^{i})}{\prod_{i=0}^{N-1}\prod_{j=1}^{i}\tilde{\Pi}_{q,t;u}(X^{i};Y^{j})}.
\end{equation}
Let us first write the weight by means of matrix elements of vertex operators:
\begin{equation}
	W^{\bm{X},\bm{Y}}_{q,t;u}(\bm{\lambda},\bm{\mu})=u^{|\mu^{N}|}\prod_{i=1}^{N}\braket{Q_{\lambda^{i}}(q,t)|\Gamma (Y^{i})_{-}|P_{\mu^{i}}(q,t)}\braket{Q_{\mu^{i}}(q,t)|\Gamma (X^{i})_{+}|P_{\lambda^{i+1}}(q,t)}.
\end{equation}
Recalling the property $\mrm{Id}_{\mcal{F}}=\sum_{\lambda\in\mbb{Y}}\ket{P_{\lambda}(q,t)}\bra{Q_{\lambda}(q,t)}$ and the definition of the degree operator $D$, we can see that
\begin{equation}
	\Pi_{q,t;u}(\bm{X};\bm{Y})=\mrm{Tr}_{\mcal{F}}\left(u^{D}\Gamma (X^{0})_{+}\Gamma (Y^{1})_{-}\Gamma (X^{1})_{+}\cdots \Gamma (Y^{N-1})_{-}\Gamma (X^{N-1})_{+}\Gamma (Y^{N})_{-}\right).
\end{equation}
To apply the trace formula in Proposition \ref{prop:trace_formula}, we rearrange the operators in the trace in the normally ordered manner.
By a standard computation of OPE for vertex operators, Proposition \ref{prop:OPE}, we have
\begin{align}
	\Gamma (X)_{+}\Gamma (Y)_{-}
	&=\exp\left(\sum_{n>0}\frac{1-t^{n}}{1-q^{n}}\frac{p_{n}(X)p_{n}(Y)}{n}\right)\Gamma (Y)_{-}\Gamma (X)_{+} \\
	&=\tilde{\Pi}_{q,t;0}(X;Y)\Gamma (Y)_{-}\Gamma (X)_{+}. \notag
\end{align}
Repeating this kind of reordering, we can make the operators normally ordered so that
\begin{equation}
	\Gamma (X^{0})_{+}\Gamma (Y^{1})_{-}\Gamma (X^{1})_{+}\cdots \Gamma (Y^{N-1})_{-}\Gamma (X^{N-1})_{+}\Gamma (Y^{N})_{-}
	=\left(\prod_{i<j}\tilde{\Pi}_{q,t;0}(X^{i};Y^{j})\right) V(\bm{\gamma}),
\end{equation}
where we set
\begin{align}
	\gamma_{n}&=\frac{1-t^{n}}{1-q^{n}}\sum_{i=0}^{N-1}p_{n}(X^{i}), & \gamma_{-n}&=\frac{1-t^{n}}{1-q^{n}}\sum_{i=1}^{N}p_{n}(Y^{i}),\quad  n>0.
\end{align}
Now we can apply Proposition \ref{prop:trace_formula} to obtain
\begin{align}
	\mrm{Tr}_{\mcal{F}}\left(u^{D}V(\bm{\gamma})\right)
	&=\frac{1}{(u;u)_{\infty}}\prod_{i=0}^{N-1}\prod_{j=1}^{N}\exp\left(\sum_{n>0}\frac{1-t^{n}}{1-q^{n}}\frac{u^{n}}{1-u^{n}}\frac{p_{n}(X^{i})p_{n}(Y^{j})}{n}\right).
\end{align}
Notice that
\begin{align}
	&\tilde{\Pi}_{q,t;0}(X;Y)\exp\left(\sum_{n>0}\frac{1-t^{n}}{1-q^{n}}\frac{u^{n}}{1-u^{n}}\frac{p_{n}(X)p_{n}(Y)}{n}\right) \\
	&=\exp\left(\sum_{n>0}\frac{1-t^{n}}{1-q^{n}}\frac{1}{1-u^{n}}\frac{p_{n}(X)p_{n}(Y)}{n}\right)  \notag \\
	&=\tilde{\Pi}_{q,t;u}(X;Y) \notag
\end{align}
Therefore, we obtain the desired result.
\end{proof}

Let us investigate some limiting cases.
\begin{itemize}
\item 	When $N=1$, we have
		\begin{align}
			\sum_{\lambda,\mu\in\mbb{Y}}u^{|\lambda|}P_{\mu/\lambda}(X)Q_{\mu/\lambda}(Y)
			=\Pi_{q,t;u}(X;Y):=\frac{1}{(u;u)_{\infty}}\prod_{i,j=1}^{\infty}\frac{(tx_{i}y_{j};q,u)_{\infty}}{(x_{i}y_{j};q,u)_{\infty}},
		\end{align}
		which was proved in \cite{RainsWarnaar2018}.
\item 	At the limit $u\to 0$, it reduces to the partition function for a Macdonald process \cite{BorodinCorwin2014,BorodinCorwinGorinShakirov2016}:
		\begin{equation}
			\Pi_{q,t;0}(\bm{X};\bm{Y})=\prod_{i<j}\tilde{\Pi}_{q,t;0}(X^{i};Y^{j}).
		\end{equation}
\item 	At the Schur-limit $q\to t$, the partition function for a periodic Schur process \cite{Borodin2007} is recovered:
		\begin{equation}
			\tilde{\Pi}_{t,t;u}(X;Y)=\prod_{i,j\ge 1}\frac{1}{(x_{i}y_{j};u)_{\infty}}.
		\end{equation}
\end{itemize}

\section{Stationary periodic Macdonald Plancherel process}
\label{eq:Plancherel_process}
In this section, we introduce a stationary periodic Macdonald Plancherel process
and show that its marginal laws give examples of periodic Macdonald processes.
The description below goes along the line of \cite[Section 6]{BeteaBouttier2019} for periodic Schur processes
except that we employ the Fock space $\mcal{F}$, for which the pairing with the dual $\mcal{F}^{\dagger}$ is $(q,t)$-deformed,

For a positive number $\xi>0$, we write $\rho_{\xi}$ for the Plancherel specialization of parameter $\xi$, i.e., it is the specialization $\rho_{\xi}:\Lambda\to\mbb{R}$ defined by
\begin{equation}
	\rho_{\xi}(p_{n})=\xi \delta_{n,1},\quad n\in\mbb{N}.
\end{equation}

\begin{prop}
For $\xi>0$, the Plancherel specialization of the Macdonald symmetric functions are given by
\begin{align}
	P_{\lambda/\mu}(\rho_{\xi};q,t)&=\frac{\xi^{|\lambda|-|\mu|}}{(|\lambda|-|\mu|)!}\dim_{q,t} (\mu,\lambda), & Q_{\lambda/\mu}(\rho_{\xi};q,t)=\frac{\xi^{|\lambda|-|\mu|}}{(|\lambda|-|\mu|)!}\dim_{q,t}^{\prime} (\mu,\lambda).
\end{align}
Here, the functions $\dim_{q,t} (\mu,\lambda)$ and $\dim_{q,t}^{\prime}(\mu,\lambda)$ are defined by
\begin{align}
	\dim_{q,t} (\mu,\lambda)&=\sum_{\mu\nearrow \nu_{1}\nearrow\cdots\nearrow \nu_{|\lambda|-|\mu|-1}\nearrow\lambda}\prod_{i=0}^{|\lambda|-|\mu|-1}\psi_{\nu_{i+1}/\nu_{i}}(q,t), \\
	\dim_{q,t}^{\prime} (\mu,\lambda)&=\sum_{\mu\nearrow \nu_{1}\nearrow\cdots\nearrow \nu_{|\lambda|-|\mu|-1}\nearrow\lambda}\prod_{i=0}^{|\lambda|-|\mu|-1}\varphi_{\nu_{i+1}/\nu_{i}}(q,t),
\end{align}
where the sum runs over paths from $\mu$ to $\lambda$ in the Young graph and we wrote $\nu_{0}=\mu$, $\nu_{|\lambda|-|\mu|}=\lambda$.
\end{prop}
\begin{proof}
The Pieri rules say
\begin{align}
	P_{\lambda}(q,t)g_{1}(q,t)&=\sum_{\mu; \lambda\nearrow\mu}\varphi_{\mu/\lambda}(q,t)P_{\mu}(q,t), &
	Q_{\lambda}(q,t)g_{1}(q,t)&=\sum_{\mu; \lambda\nearrow\mu}\psi_{\mu/\lambda}(q,t)Q_{\mu}(q,t).
\end{align}
Noting $g_{1}(q,t)=\frac{1-t}{1-q}p_{1}$, we have
\begin{align}
	\frac{1-t}{1-q}a_{-1}\ket{P_{\lambda}(q,t)}&=\sum_{\mu; \lambda\nearrow\mu}\varphi_{\mu/\lambda}(q,t)\ket{P_{\mu}(q,t)},\\
	\frac{1-t}{1-q}a_{-1}\ket{Q_{\lambda}(q,t)}&=\sum_{\mu; \lambda\nearrow\mu}\psi_{\mu/\lambda}(q,t)\ket{Q_{\mu}(q,t)}.
\end{align}
On the other hand, the Plancherel specialization of the Macdonald symmetric functions is computed as
\begin{align}
	P_{\lambda/\mu}(\rho_{\xi};q,t)&=\braket{P_{\lambda}(q,t)|\Gamma (\rho_{\xi})_{-}|Q_{\mu}(q,t)}, & Q_{\lambda/\mu}(\rho_{\xi};q,t)&=\braket{Q_{\lambda}(q,t)|\Gamma (\rho_{\xi})_{-}|P_{\mu}(q,t)},
\end{align}
where
\begin{equation}
	\Gamma (\rho_{\xi})_{-}=\exp\left(\xi \frac{1-t}{1-q}a_{-1}\right).
\end{equation}
Then, the desired results can be verified combinatorially.
\end{proof}

For parameters $\gamma>0$ and $u\in (0,1)$, we define an operator on $\mcal{F}$
\begin{equation}
	\mcal{T}^{(\gamma)}(u):=e^{\frac{1-t}{1-q}\gamma^{2} (u-1)}\Gamma (\rho_{\gamma (1-u)})_{-}u^{D}\Gamma (\rho_{\gamma (1-u)})_{+}.
\end{equation}

\begin{prop}
For a fixed $\gamma>0$ and parameters $u,v\in (0,1)$, we have
\begin{equation}
	\mcal{T}^{(\gamma)}(u)\mcal{T}^{(\gamma)}(v)=\mcal{T}^{(\gamma)}(uv).
\end{equation}
Moreover, $\mrm{Tr}_{\mcal{F}}(\mcal{T}^{(\gamma)}(u))=1/(u;u)_{\infty}$.
\end{prop}
\begin{proof}
This is verified by a direct computation:
\begin{align}
	\mcal{T}^{(\gamma)}(u)\mcal{T}^{(\gamma)}(v)
	=&e^{\frac{1-t}{1-q}\gamma^{2}(u+v-2)}\Gamma (\rho_{\gamma (1-u)})_{-}u^{D}\Gamma (\rho_{\gamma (1-u)})_{+}\Gamma (\rho_{\gamma (1-v)})_{-} v^{D}\Gamma (\rho_{\gamma (1-v)})_{+} \\
	=&e^{\frac{1-t}{1-q}\gamma^{2}(uv-1)}\Gamma (\rho_{\gamma (1-u)})_{-}u^{D}\Gamma (\rho_{\gamma (1-v)})_{-}\Gamma (\rho_{\gamma (1-u)})_{+} v^{D}\Gamma (\rho_{\gamma (1-v)})_{+} \notag \\
	=&e^{\frac{1-t}{1-q}\gamma^{2}(uv-1)}\Gamma (\rho_{\gamma (1-u)})_{-}\Gamma (\rho_{u\gamma (1-v)})_{-}(uv)^{D}\Gamma (\rho_{v\gamma (1-u)})_{+} \Gamma (\rho_{\gamma (1-v)})_{+} \notag \\
	=&e^{\frac{1-t}{1-q}\gamma^{2}(uv-1)}\Gamma (\rho_{\gamma (1-uv)})_{-}(uv)^{D}\Gamma (\rho_{\gamma (1-uv)})_{+}. \notag
\end{align}
Therefore, the desired property holds.
\end{proof}

For $\lambda,\mu\in\mbb{Y}$, we write matrix elements of $\mcal{T}^{(\gamma)}(u)$ as
\begin{equation}
	T^{(\gamma)}_{\lambda\mu}(u)=\braket{P_{\lambda}|\mcal{T}^{(\gamma)}(u)|Q_{\mu}}, \quad \lambda,\mu\in\mbb{Y}.
\end{equation}

\begin{defn}
Let $\gamma>0$ and $\beta>0$.
The stationary $\beta$-periodic Macdonald Plancherel process of intensity $\gamma$ is a stochastic process $(\lambda (t):t\in\mbb{R})$ in $\mbb{Y}$
such that $\lambda (t+\beta)=\lambda (t)$, $t\in\mbb{R}$ a.s. whose finite dimensional reduction measure for each $0=b_{0}<b_{1}<\cdots <b_{n}<b_{n+1}=\beta $ is given by
\begin{equation}
	\mrm{Prob}\left(\lambda (b_{i})=\lambda^{i},i=0,1,\dots, n\right)=(e^{-\beta};e^{-\beta})_{\infty}\prod_{i=0}^{n}T^{(\gamma)}_{\lambda^{i},\lambda^{i+1}}(e^{-(b_{i+1}-b_{i})}).
\end{equation}
\end{defn}

\begin{prop}
Suppose that $(\lambda (t):t\in\mbb{R})$ obeys the law of the stationary $\beta$-periodic Macdonald Plancherel process of intensity $\gamma$ and take a sequence $0=b_{0}<b_{1}<\cdots <b_{n}<b_{n+1}=\beta$.
Let $\bm{\rho}^{+}=(\rho^{+}_{0},\dots, \rho^{+}_{n})$ and $\bm{\rho}^{-}=(\rho^{-}_{1},\dots, \rho^{-}_{n+1})$ be the sequences of Plancherel specializations defined by $\rho^{+}_{0}=\rho_{\gamma (1-e^{b_{n}-\beta})}$, $\rho^{+}_{i}=\rho_{\gamma (e^{b_{i}}-e^{b_{i-1}})}$, $i=1,\dots, n$, $\rho^{-}_{i}=\rho_{\gamma (e^{-b_{i-1}}-e^{-b_{i}})}$, $i=1,\dots, n+1$.
Then the marginal law for $\lambda (b_{i})$, $i=0,1,\dots, n$ is the corresponding $(n+1)$-step periodic Macdonald process at $u=e^{-\beta}$:
\begin{equation}
	\mrm{Prob}\left(\lambda (b_{i})=\lambda^{i},i=0,1,\dots, n\right)=\mbb{P}^{\bm{\rho}^{+},\bm{\rho^{-}}}_{q,t;e^{-\beta}}(\lambda^{0},\lambda^{1},\dots ,\lambda^{n}).
\end{equation}
\end{prop}
\begin{proof}
The marginal probability law of interest is proportional to
\begin{align}
	\mrm{Tr}_{\mcal{F}}\Bigl(& e^{-(\beta-b_{n})D}\Gamma (\rho_{\gamma (1-e^{-(\beta-b_{n})})})_{+}\ket{Q_{\lambda^{0}}}\bra{P_{\lambda^{ 0}}}\Gamma (\rho_{\gamma (1-e^{-b_{1}})})_{-} \\
	&\times e^{-b_{1}D}\Gamma (\rho_{\gamma (1-e^{-b_{1}})})_{+}\ket{Q_{\lambda^{1}}}\bra{P_{\lambda^{1}}}\Gamma (\rho_{\gamma (1-e^{-(b_{2}-b_{1})})})_{-} \notag \\
	&\times e^{-(b_{2}-b_{1})D}\Gamma (\rho_{\gamma (1-e^{-(b_{2}-b_{1})})})_{+}\ket{Q_{\lambda^{2}}}\bra{P_{\lambda^{2}}}\Gamma (\rho_{\gamma (1-e^{-(b_{3}-b_{2})})})_{-} \notag \\
	&\cdots \times e^{-(b_{n}-b_{n-1})D}\Gamma (\rho_{\gamma (1-e^{-(b_{n}-b_{n-1})})})_{+}\ket{Q_{\lambda^{n}}}\bra{P_{\lambda^{n}}}\Gamma (\rho_{\gamma (1-e^{-(\beta-b_{n})})})_{-}\Bigr). \notag
\end{align}
Recall that the degree operator $D$ commutes with a projection $\ket{Q_{\lambda}}\bra{P_{\lambda}}$. By moving the exponentiated degree operators except for the one in the first line to the right and using the cyclic property of the trace, we see that it is identical to
\begin{align}
	\mrm{Tr}_{\mcal{F}}\Bigl(& e^{-\beta D}\Gamma (\rho_{\gamma (1-e^{-(\beta-b_{n})})})_{+}\ket{Q_{\lambda^{0}}}\bra{P_{\lambda^{0}}}\Gamma (\rho_{\gamma (1-e^{-b_{1}})})_{-} \\
	&\times \Gamma (\rho_{\gamma (e^{b_{1}}-1)})_{+}\ket{Q_{\lambda^{1}}}\bra{P_{\lambda^{1}}}\Gamma (\rho_{\gamma (e^{-b_{1}}-e^{-b_{2})})})_{-} \notag \\
	&\times \Gamma (\rho_{\gamma (e^{b_{2}}-e^{b_{1}})})_{+}\ket{Q_{\lambda^{2}}}\bra{P_{\lambda^{2}}}\Gamma (\rho_{\gamma (e^{-b_{2}}-e^{-b_{3}})})_{-} \notag \\
	&\cdots \times \Gamma (\rho_{\gamma (e^{b_{n}}-e^{b_{n-1}})})_{+}\ket{Q_{\lambda^{n}}}\bra{P_{\lambda^{n}}}\Gamma (\rho_{\gamma (e^{-b_{n}}-e^{-\beta})})_{-}\Bigr), \notag
\end{align}
which is proportional to the desired $(n+1)$-step periodic Macdonald process.
\end{proof}

\section{Observables of periodic Macdonald measures}
\label{sect:observables}
For a random variable $f:\mbb{Y}\to\mbb{F}$, we define the operator $\mcal{O}(f)\in\mrm{End}(\mcal{F})$ by
\begin{equation}
	\mcal{O}(f):=\sum_{\lambda\in\mbb{Y}}f(\lambda)\ket{P_{\lambda}}\bra{Q_{\lambda}},
\end{equation}
where the function $f$ is identified with the spectrum of the operator $\mcal{O}(f)$.
We also write, for Macdonald positive specializations $\rho^{+}$ and $\rho^{-}$,
\begin{equation}
	\psi^{\rho^{+},\rho^{-}}(f):=\Gamma (\rho^{+})_{+}\mcal{O}(f)\Gamma (\rho^{-})_{-}.
\end{equation}
Then, the moment of random variables $f_{1},\dots, f_{N}$ under an $N$-step periodic Macdonald process $\mbb{P}^{\bm{\rho}^{+},\bm{\rho}^{-}}_{q,t;u}$ is identified with the following trace over the Fock space $\mcal{F}$:
\begin{equation}
\label{eq:expectation_free_field}
	\mbb{E}^{\bm{\rho}^{+},\bm{\rho}^{-}}_{q,t;u}[f[1]\cdots f[N]]=\frac{\mrm{Tr}_{\mcal{F}}\left(u^{D}\psi^{\rho^{+}_{0},\rho^{-}_{1}}(f_{1})\cdots \psi^{\rho^{+}_{N-1},\rho^{-}_{N}}(f_{N})\right)}{\mrm{Tr}_{\mcal{F}}\left(u^{D}\psi^{\rho^{+}_{0},\rho^{-}_{1}}(1)\cdots \psi^{\rho^{+}_{N-1},\rho^{-}_{N}}(1)\right)},
\end{equation}
where $1$ is the unit function $1(\lambda)=1$, $\lambda\in\mbb{Y}$.

\subsection{Proof of Theorem \ref{thm:first_series}: First series of observables}
\label{subsect:first_series}
Comparing the values of the random variable $\mcal{E}_{r}$ and the eigenvalues of the $r$-th Macdonald operator in Sect.\ \ref{sect:preliminaries}, we can conclude that $\mcal{O}(\mcal{E}_{r})=t^{r}\what{E}_{r}$.

\begin{proof}[Proof of Theorem \ref{thm:first_series}]
The proof is very similar to that presented in \cite{SK2019} except that the vacuum expectation value there is replaced by the trace over the Fock space here.
From Proposition \ref{prop:free_field_Macdonald_operator}, it follows that
\begin{equation}
	\mcal{O}(\mcal{E}_{r})=\frac{1}{r!}\int D^{r}\bm{z}\det\left(\frac{1}{z_{i}-t^{-1}z_{j}}\right)_{1\le i,j\le r}\no{\eta(z_{1})\cdots \eta (z_{r})}.
\end{equation}
Now, the product of operators $\psi^{\rho^{+}_{0},\rho^{-}_{1}}(\mcal{E}_{r_{1}})\cdots \psi^{\rho^{+}_{N-1},\rho^{-}_{N}}(\mcal{E}_{r_{N}})$ is not normally ordered with respect to the Heisenberg algebra. The first step is to rearrange it into the normally ordered product.
By a standard computation of OPE, Proposition \ref{prop:OPE}, we have
\begin{align}
	&\psi^{\rho^{+}_{0},\rho^{-}_{1}}(\mcal{E}_{r_{1}})\cdots \psi^{\rho^{+}_{N-1},\rho^{-}_{N}}(\mcal{E}_{r_{N}}) \\
	&=\frac{1}{\bm{r}!}\int D^{\bm{r}}\bm{z} \prod_{a=1}^{N}\det\left(\frac{1}{z^{(a)}_{i}-t^{-1}z^{(a)}_{j}}\right)_{1\le i,j\le r_{a}}\prod_{a<b}\tilde{\Pi}_{q,t;u}(\rho^{+}_{a};\rho^{-}_{b}) \notag \\
	&\hspace{50pt}\times\prod_{a<b}\exp\left(\sum_{n>0}\frac{1-t^{-n}}{n}p_{n}(\rho^{+}_{a})\sum_{i=1}^{r_{b}}(z^{(b)}_{i})^{n}\right) \notag \\
	&\hspace{50pt}\times \prod_{a\ge b}\exp\left(-\sum_{n>0}\frac{1-t^{n}}{n}p_{n}(\rho^{-}_{a})\sum_{i=1}^{r_{b}}(z^{(b)}_{i})^{-n}\right) \notag \\
	&\hspace{50pt}\times \prod_{a<b}\prod_{i=1}^{r_{a}}\prod_{j=1}^{r_{b}}\frac{(1-z^{(b)}_{j}/z^{(a)}_{i})(1-qt^{-1}z^{(b)}_{j}/z^{(a)}_{i})}{(1-qz^{(b)}_{j}/z^{(a)}_{i})(1-t^{-1}z^{(b)}_{j}/z^{(a)}_{i})} V(\bm{\gamma}), \notag
\end{align}
where
\begin{align}
	\gamma_{n}&=-(1-t^{n})\sum_{a=1}^{N}\sum_{i=1}^{r_{a}}(z^{(a)}_{i})^{-n}+\frac{1-t^{n}}{1-q^{n}}\sum_{a=0}^{N-1}p_{n}(\rho^{+}_{a}), \\
	\gamma_{-n}&=(1-t^{-n})\sum_{a=1}^{N}\sum_{i=1}^{r_{a}}(z^{(a)}_{i})^{n}+\frac{1-t^{n}}{1-q^{n}}\sum_{a=1}^{N}p_{n}(\rho^{-}_{a})
\end{align}
for $n>0$.
Applying the trace formula for a vertex operator in Proposition \ref{prop:trace_formula}, we can see that the desired result follows.
\end{proof}

We saw in Corollary \ref{cor:expectation_first_series} that, when $N=1$, the expectation value gets simpler.
Let us see a particular case where the specializations $\rho^{+}$ and $\rho^{-}$ are the same Plancherel specialization $\rho_{\gamma (1-u)}$ for $\gamma>0$ as in Sect.\ \ref{eq:Plancherel_process}.
Then, the kernel of the determinant becomes
\begin{equation}
	K_{q,t;u;\mcal{E}}^{\rho_{\gamma (1-u)},\rho_{\gamma (1-u)}}(z,w)=\frac{e^{\gamma (1-t^{-1})z- \gamma (1-t)z^{-1}}}{z-t^{-1}w}.
\end{equation}

\begin{exam}
Let us compute the specific case $r=1$.
In this case,
\begin{align}
	\mbb{E}^{\rho_{\gamma (1-u)},\rho_{\gamma (1-u)}}_{q,t;u}[\mcal{E}_{1}]&=\frac{1}{1-t^{-1}}\frac{(u;u)_{\infty}(qt^{-1}u;u)_{\infty}}{(qu;u)_{\infty}(t^{-1}u;u)_{\infty}}\int \frac{dz}{2\pi\sqrt{-1}z}e^{\gamma (1-t^{-1})z-\gamma (1-t)z^{-1}} \\
	&=\frac{1}{1-t^{-1}}\frac{(u;u)_{\infty}(qt^{-1}u;u)_{\infty}}{(qu;u)_{\infty}(t^{-1}u;u)_{\infty}}\sum_{n=0}^{\infty}\frac{(\gamma^{2}(t^{-1}-1)(1-t))^{n}}{(n!)^{2}}. \notag
\end{align}
Recall that the $0$-th order modified Bessel function of the first kind is defined by
\begin{equation}
	I_{0}(z)=\sum_{n=0}^{\infty}\frac{(z^{2}/4)^{n}}{(n!)^{2}}.
\end{equation}
Then, the desired expectation value is expressed as
\begin{equation}
	\mbb{E}^{\rho_{\gamma (1-u)},\rho_{\gamma (1-u)}}_{q,t;u}[\mcal{E}_{1}]=\frac{I_{0}(2\gamma\sqrt{(1-t)(t^{-1}-1)})}{1-t^{-1}}\frac{(u;u)_{\infty}(qt^{-1}u;u)_{\infty}}{(qu;u)_{\infty}(t^{-1}u;u)_{\infty}}.
\end{equation}

Let us further take the Hall--Littlewood limit $q\to 0$. Under this limit, we have
\begin{equation}
	\mcal{E}_{1}(\lambda)\to \frac{t^{-\pr{\lambda}_{1}}}{1-t^{-1}}.
\end{equation}
Therefore, we can see that
\begin{equation}
	\mbb{E}^{\rho_{\gamma (1-u)},\rho_{\gamma (1-u)}}_{0,t;u}\left[t^{-\pr{\lambda}_{1}}\right]=I_{0}(2\gamma\sqrt{(1-t)(t^{-1}-1)})\frac{(u;u)_{\infty}}{(t^{-1}u;u)_{\infty}}.
\end{equation}
\end{exam}

In the subsequent subsections, we focus on a periodic Macdonald measure; the case of $N=1$. It is not difficult to see that the results below can be extended to an $N$-step periodic Macdonald processes of $N\ge 2$.
\subsection{Second series of observables}
\label{subsect:second_series}
The next series of observables is obtained from the first one by inverting parameters $q$ and $t$.
For $r\in\mbb{N}$, we consider the random variable $\mcal{E}^{\prime}_{r}:\mbb{Y}\to\mbb{F}$ defined by
$\mcal{E}^{\prime}_{r}(\lambda)=e_{r}(q^{-\lambda}t^{\rho -1})$, $\lambda\in\mbb{Y}$.
Then the corresponding operator is $\mcal{O}(\mcal{E}^{\prime}_{r})=t^{-r}\what{E}^{\prime}_{r}$, $r\in\mbb{N}$.

\begin{thm}
\label{thm:second_series}
For $r\in\mbb{N}$, the expectation value of $\mcal{E}^{\prime}_{r}$ under a periodic Macdonald measure $\mbb{P}^{\rho^{+},\rho^{-}}_{q,t;u}$ is computed as
\begin{align}
	\mbb{E}^{\rho^{+},\rho^{-}}_{q,t;u}[\mcal{E}^{\prime}_{r}]
	=&\frac{1}{r!}\int D^{r}\bm{z} \det\left(K^{\rho^{+},\rho^{-}}_{q,t;u;\mcal{E}^{\prime}}(z_{i},z_{j})\right)_{1\le i,j\le r}\Delta_{q^{-1},t;u}(\bm{z}),
\end{align}
where
\begin{align}
	K^{\rho^{+},\rho^{-}}_{q,t;u;\mcal{E}^{\prime}}(z,w)=\frac{1}{z-tw}\prod_{n>0}\exp\Biggl(&-\frac{1-t^{-n}}{1-u^{n}}\frac{(t/q)^{n/2}p_{n}(\rho^{+})}{n}z^{n}\\
	&+\frac{1-t^{n}}{1-u^{n}}\frac{(t/q)^{n/2}p_{n}(\rho^{-})}{n}z^{-n}\Biggr). \notag
\end{align}
\end{thm}
\begin{proof}
We write $\no{\xi (z_{1})\cdots \xi (z_{r})}=\xi (\bm{z})_{-}\xi (\bm{z})_{+}$, where
\begin{equation}
	\xi (\bm{z})_{\pm}=\exp\left(\pm \sum_{n>0}\frac{1-t^{\pm n}}{n}(t/q)^{n/2}a_{\pm n}\sum_{i=1}^{r}z_{i}^{\mp n}\right).
\end{equation}
Then, by a standard computation of OPE, Proposition \ref{prop:OPE}, we have
\begin{align}
	&\Gamma (\rho^{+})_{+}\xi (\bm{z})_{-} \xi (\bm{z})_{+} \Gamma (\rho^{-})_{-} \\
	&=\tilde{\Pi}_{q,t;0}(\rho^{+};\rho^{-})\exp\left(-\sum_{n>0}\frac{1-t^{-n}}{n}(t/q)^{n/2}p_{n}(\rho^{+})\sum_{i=1}^{r}z_{i}^{n}\right)  \notag\\
	&\hspace{70pt}\times \exp\left(\sum_{n>0}\frac{1-t^{n}}{n}(t/q)^{n/2}p_{n}(\rho^{-})\sum_{i=1}^{r}z_{i}^{-n}\right) \notag \\
	&\hspace{70pt} \times \xi (\bm{z})_{-}\Gamma (\rho^{-})_{-}\Gamma (\rho^{+})_{+} \xi (\bm{z})_{+}.\notag
\end{align}
We can see that the operator in the right hand side is
\begin{equation}
	V (\bm{\gamma})=\xi (\bm{z})_{-}\Gamma (\rho^{-})_{-}\Gamma (\rho^{+})_{+} \xi (\bm{z})_{+},
\end{equation}
where
\begin{align}
	\gamma_{n}&=(1-t^{n})(t/q)^{n/2}\sum_{i=1}^{r}z_{i}^{-n}+\frac{1-t^{n}}{1-q^{n}}p_{n}(\rho^{+}), \\
	\gamma_{-n}&=-(1-t^{-n})(t/q)^{n/2}\sum_{i=1}^{r}z_{i}^{n}+\frac{1-t^{n}}{1-q^{n}}p_{n}(\rho^{-})
\end{align}
for $n>0$.
The trace formula in Proposition \ref{prop:trace_formula} gives the desired result.
\end{proof}

If we take the Plancherel specialization $\rho^{\pm}=\rho_{\gamma (1-u)}$ with $\gamma>0$, the kernel function becomes
\begin{equation}
	K^{\rho_{\gamma (1-u)},\rho_{\gamma (1-u)}}_{q,t;u;\mcal{E}^{\prime}}(z,w)=\frac{e^{\gamma (t/q)^{1/2}\left((t^{-1}-1)z+(1-t)z^{-1}\right)}}{z-tw}.
\end{equation}

\begin{exam}
In the case of $r=1$, we have
\begin{align}
	\mbb{E}^{\rho_{\gamma (1-u)},\rho_{\gamma (1-u)}}_{q,t;u}[\mcal{E}^{\prime}_{1}]=\frac{I_{0}(2\gamma (1-t)\sqrt{q^{-1}})}{1-t}\frac{(u;u)_{\infty}(tq^{-1}u;u)_{\infty}}{(tu;u)_{\infty}(q^{-1}u;u)_{\infty}}.
\end{align}
In the $q$-Whittaker limit, $t\to 0$, we have
\begin{equation}
	\mcal{E}^{\prime}_{1}(\lambda)=\sum_{i\ge 1}q^{-\lambda_{i}}t^{i-1}\to q^{-\lambda_{1}}.
\end{equation}
Therefore,
\begin{equation}
	\mbb{E}^{\rho_{\gamma (1-u)},\rho_{\gamma (1-u)}}_{q,t;u}[q^{-\lambda_{1}}]=I_{0}(2\gamma\sqrt{q^{-1}})\frac{(u;u)_{\infty}}{(q^{-1}u;u)_{\infty}}.
\end{equation}
\end{exam}

\subsection{Third series of observables}
For $r\in\mbb{N}$, we consider $\mcal{G}_{r}:\mbb{Y}\to\mbb{F}$ defined by $\mcal{G}_{r}(\lambda)=g_{r}(q^{\lambda}t^{-\rho};q,t)$.
Following the argument in Subsect.\ \ref{subsect:free_field_operators}, we see that the corresponding operator is
\begin{equation}
	\mcal{O}(\mcal{G}_{r})=\hat{G}_{r}=\frac{(-1)^{r}}{r!}\int D^{r}\bm{z}\det\left(\frac{1}{z_{i}-qz_{j}}\right)_{1\le i,j\le r} \no{\eta (z_{1})\cdots \eta (z_{r})},
\end{equation}
which is almost the same as $\mcal{O}(\mcal{E}_{r})$ except for the determinant part and the overall sign.
Therefore, the following theorem immediately follows.

\begin{thm}
For $r\in\mbb{N}$, the expectation value of $\mcal{G}_{r}$ under a periodic Macdonald measure $\mbb{P}^{\rho^{+},\rho^{-}}_{q,t;u}$ is
\begin{align}
	\mbb{E}^{\rho^{+},\rho^{-}}_{q,t;u}[\mcal{G}_{r}]
	=&\frac{(-1)^{r}}{r!}\int D^{r}\bm{z} \det\left(K^{\rho^{+},\rho^{-}}_{q,t;u;\mcal{G}}(z_{i},z_{j})\right)_{1\le i,j\le r}\Delta_{q,t^{-1};u}(\bm{z}),
\end{align}
where
\begin{equation}
	K^{\rho^{+},\rho^{-}}_{q,t;u;\mcal{G}}(z,w)=\frac{1}{z-qw}\prod_{n>0}\exp\left(\frac{1-t^{-n}}{1-u^{n}}\frac{p_{n}(\rho^{+})}{n}z^{n}-\frac{1-t^{n}}{1-u^{n}}\frac{p_{n}(\rho^{-})}{n}z^{-n}\right).
\end{equation}
\end{thm}

\subsection{Fourth series of observables}
\label{subsect:fourth_series}
The final series of observables we consider is that of $\mcal{G}^{\prime}_{r}:\mbb{Y}\to\mbb{F}$, $r\in\mbb{N}$ defined by $\mcal{G}^{\prime}_{r}(\lambda)=g_{r}(q^{-\lambda}t^{\rho};q^{-1},t^{-1})=g_{r}(q^{-\lambda+1}t^{\rho-1};q,t)$, $\lambda\in\mbb{Y}$.
The corresponding operator is
\begin{equation}
	\mcal{O}(\mcal{G}^{\prime}_{r})=\hat{G}_{r}^{\prime}=\frac{(-1)^{r}}{r!}\int D^{r}\bm{z}\det\left(\frac{1}{z_{i}-q^{-1}z_{j}}\right)_{1\le i,j\le r} \no{\xi (z_{1})\cdots \xi (z_{r})},
\end{equation}
which is almost the same as $\mcal{O}(\mcal{E}^{\prime}_{r})$.
Therefore, we have the following.
\begin{thm}
For $r\in\mbb{N}$, the expectation value of $\mcal{G}^{\prime}_{r}$ under a periodic Macdonald measure $\mbb{P}^{\rho^{+},\rho^{-}}_{q,t;u}$ is
\begin{align}
	\mbb{E}^{\rho^{+},\rho^{-}}_{q,t;u}[\mcal{G}^{\prime}_{r}]
	=&\frac{(-1)^{r}}{r!}\int D^{r}\bm{z} \det\left(K^{\rho^{+},\rho^{-}}_{q,t;u;\mcal{G}^{\prime}}(z_{i},z_{j})\right)_{1\le i,j\le r}\Delta_{q^{-1},t;u}(\bm{z}),
\end{align}
where
\begin{align}
	K^{\rho^{+},\rho^{-}}_{q,t;u;\mcal{G}^{\prime}}(z,w)=\frac{1}{z-q^{-1}w}\prod_{n>0}\exp\Biggl(&-\frac{1-t^{-n}}{1-u^{n}}\frac{(t/q)^{n/2}p_{n}(\rho^{+})}{n}z^{n}\\
	&+\frac{1-t^{n}}{1-u^{n}}\frac{(t/q)^{n/2}p_{n}(\rho^{-})}{n}z^{-n}\Biggr). \notag
\end{align}
\end{thm}

\subsection{Analytic interpretation}
As we have already pointed out in Introduction, the above results in Subsect.\,\ref{subsect:first_series}--\ref{subsect:fourth_series} are understood in the formal sense.
Here, we present the analytic interpretation for the formulas for the first and second series of observables as illustrating examples.
The idea is to carefully keep track of the region where an integrand as a formal power series converges to an analytic function.
For simplicity and convenience of description, we concentrate on the case of $N=1$ and assume that the specializations are given by
\begin{equation}
\label{eq:specialization_exam_pure_alpha}
	\rho^{+}(p_{n})=\sum_{s=1}^{\ell_{+}}\left(\alpha^{+}_{s}\right)^{n}, \quad \rho^{-}(p_{n})=\sum_{s=1}^{\ell_{-}}\left(\alpha^{-}_{s}\right)^{n}, \quad n\in \mbb{N},
\end{equation}
where the sequences $(\alpha_{s}^{\pm}\geq 0:s=1,\dots, \ell_{\pm})$ are non-increasing.
More general cases are studied in the analogous way.

\subsubsection{Second series of observables}
It is instructive first to see the analytic interpretation of the formulas in Theorem~\ref{thm:second_series} for the second series of observables.

For convenience, let us employ the change of integral variables $w_{i}:=(t/q)^{1/2}z_{i}^{-1}$, $i=1,\dots, r$.
Under the specialization (\ref{eq:specialization_exam_pure_alpha}), we find the formula
\begin{equation}
\label{eq:second-series_after_change_of_variables}
	\mbb{E}^{\rho^{+},\rho^{-}}_{q,t;u}[\pr{\mcal{E}}_{r}]=\frac{1}{r!}\int D^{r}\bm{w}\det \left[\widetilde{K}^{\rho^{+},\rho^{-}}_{q,t;u;\pr{\mcal{E}}}(w_{i},w_{j})\right]_{1\leq i,j\leq r}\Delta_{q^{-1},t;u}(\bm{w}),
\end{equation}
where
\begin{equation}
	\widetilde{K}^{\rho^{+},\rho^{-}}_{q,t;u;\pr{\mcal{E}}}(z,w)
	=\frac{1}{z-tw}H_{u}((qz)^{-1};\rho^{+})H_{u}(z;\rho^{-}),
\end{equation}
with
\begin{equation}
	H_{u}(z,\rho^{\pm}):=\prod_{s=1}^{\ell_{\pm}}\frac{(t\alpha_{s}^{\pm}z;u)_{\infty}}{(\alpha_{s}^{\pm}z;u)_{\infty}}.
\end{equation}
It is readily seen that the integrand of Eq.~(\ref{eq:second-series_after_change_of_variables}) is convergent in the region
\begin{align}
	&|tw_{i}/w_{j}|<1, \\
	&|(qw_{i})^{-1}\alpha^{+}_{1}|<1, \\
	&|w_{i}\alpha^{-}_{1}|<1, \\
	&|q^{-1}uw_{1}/w_{j}|<1, \quad i,j=1,\dots, r.
\end{align}
Therefore, if $\alpha_{1}^{+}<q$, $\alpha^{-}_{1}<1$ and $u<q$, we can take integral contours on which $|w_{i}|=1$, $i=1,\dots, r$ are satisfied.
In other words, under these conditions, we have a formula involving analytic integrals
\begin{align}
	&\mbb{E}^{\rho^{+},\rho^{-}}_{q,t;u}[\pr{\mcal{E}}_{r}]\\
	&=\frac{1}{(2\pi\sqrt{-1})^{r}r!}\int_{|w_{1}|=1}dw_{1}\cdots \int_{|w_{r}|=1}dw_{r}
	\det \left[\widetilde{K}^{\rho^{+},\rho^{-}}_{q,t;u;\pr{\mcal{E}}}(w_{i},w_{j})\right]_{1\leq i,j\leq r}\Delta_{q^{-1},t;u}(\bm{w}). \notag
\end{align}
Note that this formula recovers \cite[Proposition 3.8]{BorodinCorwinGorinShakirov2016} at the limit $u\to 0$.

\subsubsection{First series of observables}
We next analyze the formula for the first series of observables.
In this case, there appears an additional difficulty: the formal power series $(z_{i}-t^{-1}z_{j})^{-1}$ for all $i,j=1,\dots, r$ do not converge in a common region since we assumed $|t|<1$.
To overcome this difficulty, notice the following identity (see \cite[Section 4]{SK2019}):
\begin{align}
	&\frac{t^{-r}}{r!}\int D^{r}\bm{z}\det \left(\frac{1}{z_{i}-t^{-1}z_{j}}\right)_{1\le i,j\le r}F(z_{1},\dots, z_{r}) \\
	&=\frac{t^{-r(r+1)/2}}{(t^{-1};t^{-1})_{r}}\int D^{r}\bm{z}\prod_{i=1}^{r}z_{i}^{-1}\prod_{1\leq i<j\leq r}\frac{1-z_{i}/z_{j}}{1-t^{-1}z_{i}/z_{j}}F(z_{1},\dots, z_{r}) \notag
\end{align}
for any function $F(z_{1},\dots, z_{r})$ that is symmetric under permutations of $z_{i}$, $i=1,\dots, r$.
Applying the specialization (\ref{eq:specialization_exam_pure_alpha}) and the change of variables $w_{i}:=z_{i}^{-1}$, $i=1,\dots, r$, we have the formula
\begin{align}
\label{eq:first-series_after_det_broken}
	\mbb{E}^{\rho^{+},\rho^{-}}_{q,t;u}[\mcal{E}_{r}]
	=\frac{t^{-r(r-1)/2}}{(t^{-1};t^{-1})_{r}}\int D^{r}\bm{w}&\prod_{i=1}^{r}w_{i}^{-1}\prod_{1\leq i<j\leq r}\frac{1-w_{j}/w_{i}}{1-t^{-1}w_{j}/w_{i}} \\
	\times &\left(\prod_{i=1}^{r}H_{u}((tw_{i})^{-1};\rho^{+})^{-1}H_{u}(w_{i};\rho^{-})^{-1}\right)\Delta_{q,t^{-1};u}(\bm{w}) \notag.
\end{align}
Here, the integrand converges in the region
\begin{align}
	&|\alpha_{1}^{+}w_{i}^{-1}|<1,\quad |\alpha_{1}^{-}tw_{i}|<1, \quad i=1,\dots, r,\\
	&|t^{-1}w_{j}/w_{i}|<1, \quad 1\leq i< j\leq r, \\
	&|t^{-1}uw_{i}/w_{j}|<1, \quad i,j=1,\dots, r.
\end{align}
To make the analytical sense of the formula (\ref{eq:first-series_after_det_broken}), we assume that $\alpha_{1}^{+}<1$, $\alpha_{1}^{-}<t^{r}$, and $u<t^{r}$.
Then, there exists a set of contours $C_{1},\dots, C_{r}$ such that (i) $C_{i}$, $i=1,\dots, r$ encircle $\alpha_{s}^{+}$, $s=1,\dots, \ell_{+}$ and do not encircle $(\alpha_{s}^{-})^{-1}$, $s=1,\dots, \ell_{-}$, (ii) for each $i=1,\dots, r$, the area enclosed by $C_{i}$ contains $t^{-1}C_{j}$ with $j<i$, (iii) if $(w_{1},w_{r})\in C_{1}\times C_{r}$, then $|t^{-1}uw_{1}/w_{r}|<1$.
In this case, the formula (\ref{eq:first-series_after_det_broken}) admits an analytical expression that
\begin{align}
\label{eq:first-series_analytic}
	\mbb{E}^{\rho^{+},\rho^{-}}_{q,t;u}[\mcal{E}_{r}]
	=&\frac{t^{-r(r-1)/2}}{(t^{-1};t^{-1})_{r}}\frac{1}{(2\pi\sqrt{-1})^{r}}\int_{C_{1}}\frac{dw_{1}}{w_{1}}\cdots \int_{C_{r}}\frac{dw_{r}}{w_{r}}\prod_{1\leq i<j\leq r}\frac{1-w_{j}/w_{i}}{1-t^{-1}w_{j}/w_{i}} \\
	&\times \left(\prod_{i=1}^{r}H_{u}((tw_{i})^{-1};\rho^{+})^{-1}H_{u}(w_{i};\rho^{-})^{-1}\right)\Delta_{q,t^{-1};u}(\bm{w}) \notag.
\end{align}

Let us see the Hall--Littlewood limit $q\to 0$ of Eq.~(\ref{eq:first-series_analytic}).
Since we have
\begin{equation}
	\lim_{q\to 0}\mcal{E}_{r}(\lambda)=\frac{t^{-t(t-1)/2}}{(t^{-1};t^{-1})_{r}}t^{-r\pr{\lambda}_{1}}, \quad \lambda\in\mbb{Y}, \quad r\in\mbb{N},
\end{equation}
the formula (\ref{eq:first-series_analytic}) gives
\begin{align}
\label{eq:first-series_analytic_Hall--Littlewood}
	\mbb{E}^{\rho^{+},\rho^{-}}_{0,t;u}[t^{-r\pr{\lambda}_{1}}]
	=&\frac{1}{(2\pi\sqrt{-1})^{r}}\int_{C_{1}}\frac{dw_{1}}{w_{1}}\cdots \int_{C_{r}}\frac{dw_{r}}{w_{r}}\prod_{1\leq i<j\leq r}\frac{1-w_{j}/w_{i}}{1-t^{-1}w_{j}/w_{i}} \\
	&\times \left(\prod_{i=1}^{r}H_{u}((tw_{i})^{-1};\rho^{+})^{-1}H_{u}(w_{i};\rho^{-})^{-1}\right)\Delta_{0,t^{-1};u}(\bm{w}) \notag.
\end{align}
This formula (\ref{eq:first-series_analytic_Hall--Littlewood}) reduces to \cite[Proposition 3.1]{Dimitrov2018} at the limit $u\to 0$.

\section{Shift-mixed measures and the Schur-limit}
\label{app:shift-mixed}
In this section, we prove Theorem \ref{thm:expectation_shift_mixed} concerning an expectation value under a shift-mixed periodic Macdonald measure and its Schur-limit, Proposition \ref{prop:shift-mixed_Schur_limit}.

\subsection{Fock space description}
It is immediate that a shift-mixed periodic Macdonald measure admits an interpretation in terms of Fock spaces.
We set $\Xi=\mcal{F}\otimes \mbb{C}[e^{\pm \alpha}]$ and $\Xi^{\dagger}=\mcal{F}^{\dagger}\otimes \mbb{C}[e^{\pm \alpha}]$.
We often write $\ket{v\otimes e^{n\alpha}}:=\ket{v}\otimes e^{n\alpha}\in \Xi$, $\ket{v}\in\mcal{F}$, $n\in\mbb{Z}$ and
$\bra{v\otimes e^{n\alpha}}=\bra{v}\otimes e^{n\alpha}\in \Xi^{\dagger}$, $\bra{v}\in \mcal{F}^{\dagger}$, $n\in\mbb{Z}$.
The paring $\Xi^{\dagger}\times \Xi\to\mbb{C}$ is naturally defined by $\braket{v\otimes e^{m\alpha}|w\otimes e^{n\alpha}}=\braket{v|w}\delta_{m,n}$, $\bra{v}\in\mcal{F}^{\dagger}$, $\ket{w}\in\mcal{F}$, $m,n\in\mbb{Z}$.
We define the charge operator $a_{0}\in\mrm{End}(\Xi)$ by $a_{0}\ket{v\otimes e^{n\alpha}}=n\ket{v\otimes e^{n\alpha}}$, $\ket{v}\in \mcal{F}$, $n\in\mbb{Z}$ and the energy operator $H=D+a_{0}^{2}/2$.
Then, the weight of a shift-mixed periodic Macdonald measure is expressed as
\begin{equation}
	\mbb{P}^{\rho^{+},\rho^{-}}_{q,t;u,\zeta}(\lambda,n)=\frac{\mrm{Tr}_{\Xi}\left(u^{H}\zeta^{a_{0}}\Gamma (\rho^{+})_{+}\ket{P_{\lambda}\otimes e^{n\alpha}}\bra{Q_{\lambda}\otimes e^{n\alpha}}\Gamma (\rho^{-})_{-}\right)}{\mrm{Tr}_{\Xi}\left(u^{H}\zeta^{a_{0}}\Gamma (\rho^{+})_{+}\Gamma (\rho^{-})_{-}\right)}.
\end{equation}

\subsection{Observables}
We extend the symbol $\mcal{O}(f)$ for a function $f:\mbb{Y}\to \mbb{R}$ to a function on $\mbb{Y}\times\mbb{Z}$ by
\begin{equation}
	\mcal{O}(f)=\sum_{\lambda\in\mbb{Y},n\in\mbb{Z}}f(\lambda,n)\ket{P_{\lambda}\otimes e^{n\alpha}}\bra{Q_{\lambda}\otimes e^{n\alpha}}\in\mrm{End}(\Xi),\ \ f:\mbb{Y}\times\mbb{Z}\to\mbb{R}.
\end{equation}
Obviously, the operator corresponding to $\tilde{\mcal{E}}_{r}$ is 
\begin{equation}
	\mcal{O}(\tilde{\mcal{E}}_{r})=\tilde{E}_{r}:=\frac{1}{r!}\int D^{r}\bm{z}\det\left(\frac{1}{z_{i}-t^{-1}z_{j}}\right)_{1\le i,j\le r}t^{-ra_{0}}\no{\eta (z_{1})\cdots \eta (z_{r})},
\end{equation}
which we call the $r$-th extended Macdonald operator.
Therefore, the expectation value is computed as
\begin{equation}
	\mbb{E}^{\rho^{+},\rho^{-}}_{q,t;u,\zeta}[\tilde{\mcal{E}}_{r}]=\sum_{\lambda\in\mbb{Y},n\in\mbb{Z}}\tilde{\mcal{E}}_{r}(\lambda,n)\mbb{P}^{\rho^{+},\rho^{-}}_{q,t;u,\zeta}(\lambda,n)=\frac{\mrm{Tr}_{\Xi}\left(u^{H}\zeta^{a_{0}}\Gamma (\rho^{+})_{+}\tilde{E}_{r}\Gamma (\rho^{-})_{-}\right)}{\mrm{Tr}_{\Xi}\left(u^{H}\zeta^{a_{0}}\Gamma (\rho^{+})_{+}\Gamma (\rho^{-})_{-}\right)}.
\end{equation}

\begin{proof}[Proof of Theorem \ref{thm:expectation_shift_mixed}]
It follows straightforwardly that
\begin{equation}
	\mrm{Tr}_{\Xi}\left(u^{H}\zeta^{a_{0}}\Gamma (\rho^{+})_{+}\tilde{E}_{r}\Gamma (\rho^{-})_{-}\right)
	=\theta_{3}(\zeta t^{-r};u)\mrm{Tr}_{\mcal{F}}\left(u^{D}\Gamma (\rho^{+})_{+}\what{E}_{r}\Gamma (\rho^{-})_{-}\right),
\end{equation}
while the trace over the Fock space $\mcal{F}$ has been computed in the proof of Theorem \ref{thm:first_series}
\end{proof}

\subsection{Schur-limit}
We fix $t$ and consider the Schur-limit $q\to t$.
Let us introduce the charged free fermions:
\begin{align}
\label{eq:charged_free_fermion_field_1}
	\psi (z)&=e^{-\alpha}z^{-a_{0}}\exp\left(\sum_{n>0}\frac{a_{-n}}{n}z^{n}\right)\exp\left(-\sum_{n>0}\frac{a_{n}}{n}z^{-n}\right), \\
\label{eq:charged_free_fermion_field_2}	
	\psi^{\ast}(z)&= e^{\alpha}z^{a_{0}}\exp\left(-\sum_{n>0}\frac{a_{-n}}{n}z^{n}\right)\exp\left(\sum_{n>0}\frac{a_{n}}{n}z^{-n}\right).
\end{align}
According to the boson-fermion correspondence (see e.g. \cite[Lecture 5]{KacRainaRozhkovskaya2013}), the space $\Xi$ is isomorphic to the fermion Fock space with respect to the fermion fields (\ref{eq:charged_free_fermion_field_1}), (\ref{eq:charged_free_fermion_field_2}) at the Schur-limit $q\to t$.
For $r\in\mbb{N}$, set
\begin{align}
	\scr{E}_{r}(t)&=\frac{1}{r!}\int D^{r}\bm{z}\scr{E}_{r}(\bm{z},t), &
	\scr{E}_{r}(\bm{z},t)&=\psi (z_{1})\cdots \psi (z_{r})\psi^{\ast}(t^{-1}z_{r})\cdots \psi^{\ast}(t^{-1}z_{1}).
\end{align}

\begin{prop}
\label{prop:Macdonald_operator_Schur_limit}
For each $r\in\mbb{N}$, the extended Macdonald operator $\tilde{E}_{r}$ reduces, at the Schur-limit $q\to t$, to $\scr{E}_{r}(t)$.
\end{prop}
\begin{proof}
Owing to the fermionic Wick formula or the Cauchy determinant formula, we see that
\begin{equation}
	\scr{E}_{r}(\bm{z},t)=\det\left(\frac{1}{z_{i}-t^{-1}z_{j}}\right)_{1\le i,j\le r}t^{-ra_{0}}\no{\eta (z_{1})\cdots \eta (z_{r})},
\end{equation}
which implies the desired result.
\end{proof}

Note that the operator $\scr{E}_{1}(t)$ is essentially the same as one introduced in \cite{OkounkovPandharipande2006}.
Due to Proposition \ref{prop:Macdonald_operator_Schur_limit}, we can say that the Macdonald operators are deformation of the operators $\scr{E}_{r}(t)$, $r\in\mbb{N}$, each of which is realized by means of the charged free fermion.
More precisely, we apply deformation of the Heisenberg algebra only after rearranging the operator $\scr{E}_{r}(t)$ in the normally ordered manner with respect to the bosonic modes to obtain the corresponding Macdonald operator. Note that a similar observation has been given in \cite{Prochazka2019} for the Nazarov--Sklyanin operator in the Jack case \cite{NazarovSklyanin2013a,NazarovSklyanin2013b}.
Following this recognition about the Macdonald operators, we can understand the determinant found in the free field realization of the Macdonald operators, which is also the origin of the determinantal structure of Macdonald processes, as remnant of the free fermions before the deformation.

As was noticed in \cite[Remark 4.1]{BeteaBouttier2019}, we can see that the boson-fermion correspondence implies the following identity.
\begin{prop}
\label{prop:finite_temperature_determinant}
Let $x_{1},\dots, x_{r}$, $y_{1},\dots, y_{r}$ be indeterminates and let $u,\zeta\in (0,1)$ be parameters. We have
\begin{align}
	& \frac{\prod_{i<j}(x_{i}-x_{j})\prod_{i>j}(y_{i}-y_{j})}{\prod_{i,j=1}^{r}(x_{i}-x_{j})}\frac{\theta_{3}\left(\zeta\frac{y_{1}\cdots y_{r}}{x_{1}\cdots x_{r}};u\right)}{\theta_{3}(\zeta;u)}\prod_{i,j=1}^{r}\frac{(ux_{i}/x_{j};u)_{\infty}(uy_{i}/y_{j};u)_{\infty}}{(ux_{i}/y_{j};u)_{\infty}(uy_{i}/x_{j};u)_{\infty}} \\
	&=\det\left(\frac{1}{x_{i}-y_{j}}\frac{\theta_{3}(\zeta y_{j}/x_{i};u)}{\theta_{3}(\zeta;u)}\frac{(u;u)_{\infty}^{2}}{(ux_{i}/y_{j};u)_{\infty}(uy_{j}/x_{i};u)_{\infty}}\right)_{1\le i,j\le r}. \notag
\end{align}
\end{prop}
\begin{proof}
We compute the correlation function
\begin{equation}
	\braket{\psi (x_{1})\cdots \psi (x_{r}) \psi^{\ast}(y_{r})\cdots \psi^{\ast} (y_{1})}_{u,\zeta}
	:=\frac{\mrm{Tr}_{\Xi}\left(u^{H}\zeta^{a_{0}}\psi (x_{1})\cdots \psi (x_{r}) \psi^{\ast}(y_{r})\cdots \psi^{\ast}(y_{1})\right)}{\mrm{Tr}_{\Xi}\left(u^{H}\zeta^{a_{0}}\right)}
\end{equation}
in two different ways; in the bosonic and fermionic methods.
On the bosonic side, we have
\begin{equation}
	\psi (x_{1})\cdots \psi (x_{r}) \psi^{\ast} (y_{r})\cdots \psi^{\ast}(y_{1})
	=\frac{\prod_{i<j}(x_{i}-x_{j})\prod_{i>j}(y_{i}-y_{j})}{\prod_{i,j=1}^{r}(x_{i}-y_{j})}\left(\frac{y_{1}\cdots y_{r}}{x_{1}\cdots x_{r}}\right)^{a_{0}}V(\bm{\gamma}),
\end{equation}
where
\begin{align}
	\gamma_{n}&=-\frac{n}{|n|}\sum_{i=1}^{r}(x_{i}^{-n}-y_{i}^{-n}),\ \ n\in\mbb{Z}\backslash \set{0}.
\end{align}
The trace over $\Xi$ is evaluated by decomposing it into the sum over charges and the trace over $\mcal{F}$ that can be computed by applying Proposition \ref{prop:trace_formula}. Consequently, we have
\begin{align}
\label{eq:multi_point_boson}
	&\braket{\psi (x_{1})\cdots \psi (x_{r}) \psi^{\ast}(y_{r})\cdots \psi^{\ast} (y_{1})}_{u,\zeta} \\
	&=\frac{\prod_{i<j}(x_{i}-x_{j})\prod_{i>j}(y_{i}-y_{j})}{\prod_{i,j=1}^{r}(x_{i}-x_{j})}\frac{\theta_{3}\left(\zeta\frac{y_{1}\cdots y_{r}}{x_{1}\cdots x_{r}};u\right)}{\theta_{3}(\zeta;u)}\prod_{i,j=1}^{r}\frac{(ux_{i}/x_{j};u)_{\infty}(uy_{i}/y_{j};u)_{\infty}}{(ux_{i}/y_{j};u)_{\infty}(uy_{i}/x_{j};u)_{\infty}}. \notag
\end{align}

On the other hand, the fermionic Wick formula (see \cite[Appendix B]{BeteaBouttier2019}) allows us to have
\begin{equation}
\label{eq:multi_point_fermion}
	\braket{\psi (x_{1})\cdots \psi (x_{r}) \psi^{\ast}(y_{r})\cdots \psi^{\ast} (y_{1})}_{u,\zeta}
	=\det\left(\braket{\psi (x_{i}) \psi^{\ast} (y_{j})}_{u,\zeta}\right)_{1\le i,j\le r},
\end{equation}
where the two point function is obviously
\begin{equation}
	\braket{\psi (x) \psi^{\ast} (y)}_{u,\zeta}=\frac{1}{x-y}\frac{\theta_{3}(\zeta y/x;u)}{\theta_{3}(\zeta;u)}\frac{(u;u)_{\infty}^{2}}{(ux/y;u)_{\infty}(uy/x;u)_{\infty}}
\end{equation}
as a special case of (\ref{eq:multi_point_boson}) at $r=1$.
Comparing (\ref{eq:multi_point_boson}) and (\ref{eq:multi_point_fermion}), we can see the desired identity.
\end{proof}

As an application of Proposition \ref{prop:finite_temperature_determinant}, we can check that the formula in Theorem \ref{thm:expectation_shift_mixed} reduces to an integral of a single determinant at the Schur-limit $q\to t$ proving Proposition \ref{prop:shift-mixed_Schur_limit}, which is of course expected since the observable $\tilde{E}_{r}$ reduces to $\scr{E}_{r}$ written by means of fermions.

\appendix

\section{Generalized MacMahon formula for cylindric partitions}
\label{app:MacMahon_cylindric_partitions}
In this appendix, we describe an application of periodic Macdonald processes to combinatorics; a generalization of the MacMahon formula for cylindric partitions. 
The original MacMahon formula gives a generating function that enumerates plane partitions. As its $(q,t)$-deformation, a generalization of the MacMahon formula was proposed in \cite{Vuletic2009}.
On the other hand, the notion of plane partitions can be generalized to that of cylindric partitions by imposing a certain periodicity.
The MacMahon formula for cylindric partitions was given in \cite{Borodin2007} and its Hall--Littlewood extension was in \cite{CorteelSaveliefVuletic2011}.
Here, we present a further generalization; the Macdonald-level extension of the MacMahon formula for cylindric partitions.

In the first version of the manuscript of this paper that appeared on arXiv, the content of this appendix was in the main text. Later, we found that the same result had been obtained in \cite{Langer2012}. We believe, however, that including this appendix helps readers who are concerned with combinatorics.

\subsection{Cylindric partitions}
A cylindric partition is defined as a sequence of partitions subject to certain inclusion relations and a periodicity (Figure \ref{fig:cylindric_partition}).
Let us fix $N\in\mbb{N}$ and write $[1,N]:=\set{1,\dots, N}$.
A cylindric partition of periodicity $N$ and boundary profile $\mbb{M}\subset [1,N]$ is a sequence $\bm{\lambda}=(\lambda^{1},\dots,\lambda^{N})$ of partitions such that
\begin{equation}
	\begin{cases}
		\lambda^{k}\prec \lambda^{k+1}, & k\in \mbb{M},\\
		\lambda^{k}\succ \lambda^{k+1}, & k\not\in \mbb{M}
	\end{cases}
\end{equation}
for all $k=1,\dots, N$,
where we identify $\lambda^{N+1}=\lambda^{1}$.
Here, we write $\lambda\succ\mu$ for two partitions $\lambda,\mu$ if $\mu\subset\lambda$ and the skew-partition $\lambda/\mu$ is a horizontal strip, i.e., it has at most a single box in each column.
We write $\mcal{CP}(N,\mbb{M})$ for the collection of cylindric partitions of periodicity $N$ and boundary profile $\mbb{M}$.
For a cylindric partition $\bm{\lambda}\in\mcal{CP}(N,\mbb{M})$, we call the number $|\bm{\lambda}|=\sum_{k=1}^{N}|\lambda^{k}|$ its weight.

\begin{figure}[h]
\begin{center}
\includegraphics[width=.8\linewidth]{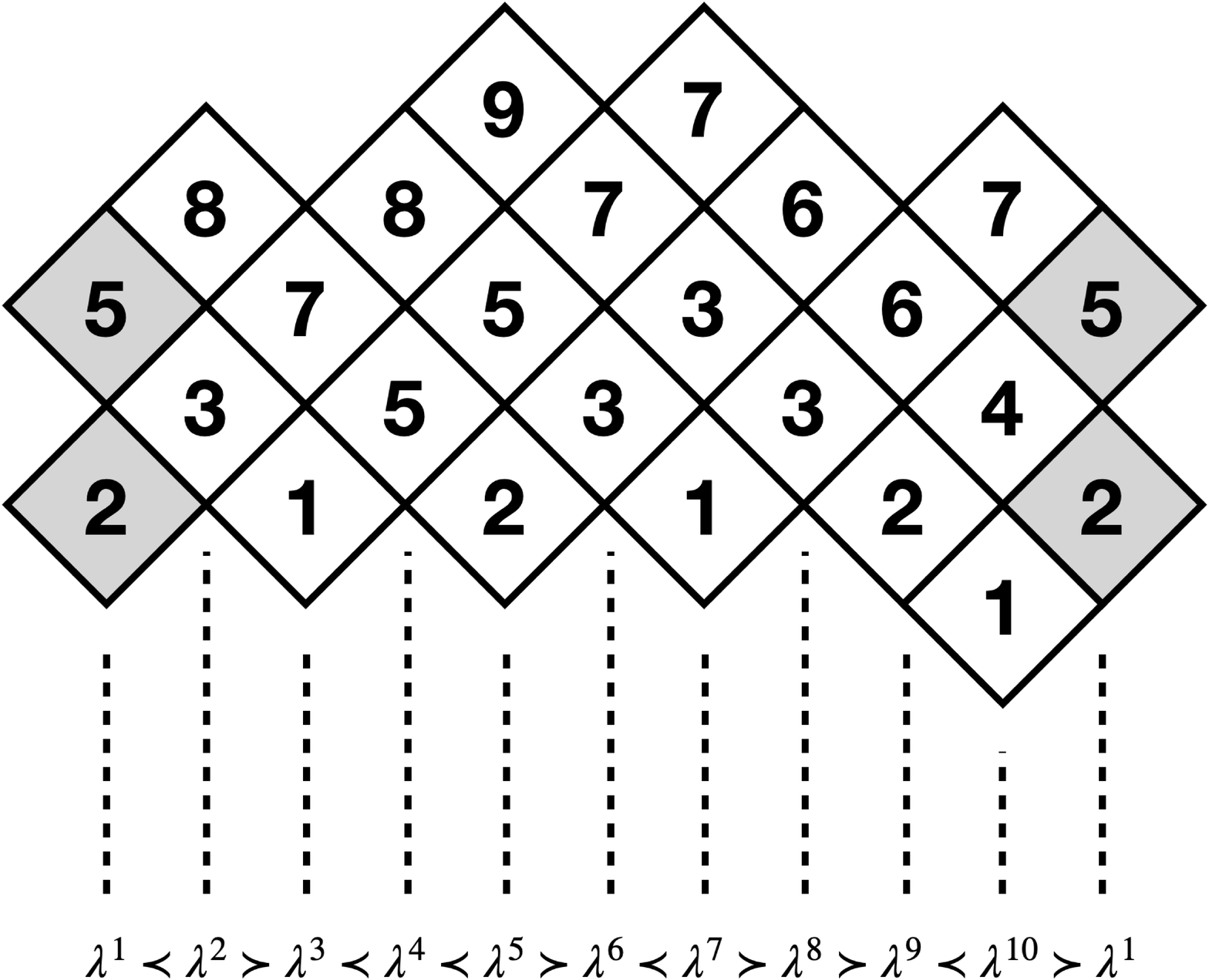}
\caption{A cylindric partition in the case that $N=10$ and $\mbb{M}=\{1,3,4,6,9\}$.}
\label{fig:cylindric_partition}
\end{center}
\end{figure}

We define a weight of a cylindric partition $\bm{\lambda}\in\mcal{CP}(N,\mbb{M})$ as follows.
To a pair of $k\in [1,N]$ and $j\in\mbb{N}$, we associate partitions
\begin{align}
\label{eq:partition_associated_to_a_box_1}
	\lambda &=(\lambda^{k}_{j},\lambda^{k}_{j+1},\dots), \\
\label{eq:partition_associated_to_a_box_2}
	\mu &= (\lambda^{k+1}_{j+\chi [k\in \mbb{M}]},\lambda^{k+1}_{j+\chi [k\in \mbb{M}]+1},\dots), \\
\label{eq:partition_associated_to_a_box_3}
	\nu &= (\lambda^{k-1}_{j+\chi [k-1\not\in \mbb{M}]},\lambda^{k-1}_{j+\chi [k-1\not\in \mbb{M}]+1},\dots ),
\end{align}
and define a weight of the box specified by $(k,j)$ as
\begin{equation}
\label{eq:weight_cylindric_partition_building_block}
	F_{\bm{\lambda}}(k,j;q,t):=\prod_{m=0}^{\infty}\frac{f(q^{\lambda_{1}-\lambda_{m+1}}t^{m})f(q^{\lambda_{1}-\lambda_{m+2}}t^{m})}{f(q^{\lambda_{1}-\mu_{m+1}}t^{m})f(q^{\lambda_{1}-\nu_{m+1}}t^{m})}, \quad f(u)=\frac{(tu;q)_{\infty}}{(qu;q)_{\infty}}.
\end{equation}
Then, the weight of $\bm{\lambda}$ is defined as the product of (\ref{eq:weight_cylindric_partition_building_block}) over all boxes:
\begin{equation}
\label{eq:weight_cylindric_partition}
	F_{\bm{\lambda}}:=\prod_{k=1}^{N}\prod_{j\in\mbb{N}}F_{\bm{\lambda}}(k,j).
\end{equation}

The generalized MacMahon formula \cite{Langer2012} for cylindric partitions reads as follows.
\begin{thm}
\label{thm:MacMahon_cylindric_partitions}
Let $N\in\mbb{N}$ and let $\mbb{M}\subset [1,N]$ be a boundary profile. Then
\begin{equation}
	\sum_{\bm{\lambda}\in\mcal{CP}(N,\mbb{M})}F_{\bm{\lambda}}(q,t)s^{|\bm{\lambda}|}=\frac{1}{(s^{N};s^{N})_{\infty}}\prod_{k\in\mbb{M},l\not\in\mbb{M}}\frac{(ts^{[l-k]};q,s^{N})_{\infty}}{(s^{[l-k]};q,s^{N})_{\infty}},
\end{equation}
where we set
\begin{equation}
	[l-k]=
	\begin{cases}
	l-k, & l>k, \\
	l-k+N, & l<k.
	\end{cases}
\end{equation}
\end{thm}

\subsection{Proof of Theorem \ref{thm:MacMahon_cylindric_partitions}}
This subsection is devoted to a proof of Theorem \ref{thm:MacMahon_cylindric_partitions}.
The idea follows \cite{Vuletic2009, CorteelSaveliefVuletic2011}; we express each weight (\ref{eq:weight_cylindric_partition}) using a weight of a periodic Macdonald process given a certain family of specializations.

To compare our convention to that adopted in \cite{Vuletic2009}, let us set, for $n,m\in\mbb{Z}_{\ge 0}$,
\begin{equation}
	\tilde{f}(n,m)=
	\begin{cases}
		\prod_{i=0}^{n-1}\frac{1-q^{i}t^{m+1}}{1-q^{i+1}t^{m}}, & n\ge 1,\\
		1, & n=0.
	\end{cases}
\end{equation}
Then we have
\begin{equation}
	\tilde{f}(n,m)=\frac{(t^{m+1};q)_{\infty}}{(t^{m}q;q)_{\infty}}\frac{1}{f(q^{n}t^{m})}.
\end{equation}
Therefore, the weight for each box (\ref{eq:weight_cylindric_partition_building_block}) is also written as
\begin{equation}
	F_{\bm{\lambda}}(k,j;q,t)=\prod_{m=0}^{\infty}\frac{\tilde{f}(\lambda_{1}-\mu_{m+1},m)\tilde{f}(\lambda_{1}-\nu_{m+1},m)}{\tilde{f}(\lambda_{1}-\lambda_{m+1},m)\tilde{f}(\lambda_{1}-\lambda_{m+2},m)},
\end{equation}
which is the same expression as its equivalent for plane partitions found in \cite{Vuletic2009}, where the function $\tilde{f}$ here is simply denoted by $f$.
On the other hand, our convention of the function $f$ is taken from \cite{Macdonald1995}.

Given a boundary profile $\mbb{M}\subset [1,N]$, we apply the following specializations to the weight $W^{\bm{\rho}^{+},\bm{\rho}^{-}}_{q,t;u}(\bm{\lambda},\bm{\mu})$ in Eq. (\ref{eq:periodic_weight}).
For $k=1,\dots, N-1$, if $k\in\mbb{M}$, then
\begin{align}
	\rho^{+}_{k}&:x_{1}=s^{-k},\ \ x_{2}=x_{3}=\cdots =0,\\
	\rho^{-}_{k}&:x_{1}=x_{2}=\cdots =0,
\end{align}
and if $k\not\in\mbb{M}$, then
\begin{align}
	\rho^{+}_{k}&:x_{1}=x_{2}=\cdots =0,\\
	\rho^{-}_{k}&:x_{1}=s^{k},\ \ x_{2}=x_{3}=\cdots =0.
\end{align}
The specializations $\rho^{+}_{0}$ and $\rho^{-}_{N}$ are defined so that if $N\in\mbb{M}$, then
\begin{align}
	\rho^{+}_{0}&:x_{1}=1,\ \ x_{2}=x_{3}=\cdots=0, \\
	\rho^{-}_{N}&:x_{1}=x_{2}=\cdots =0,
\end{align}
and if $N\not\in\mbb{M}$, then
\begin{align}
	\rho^{+}_{0}&:x_{1}=x_{2}=\cdots =0,\\
	\rho^{+}_{N}&:x_{1}=s^{N},\ \ x_{2}=x_{3}=\cdots =0.
\end{align}

\begin{lem}
Under the above specializations, the weight $W^{\bm{\rho}^{+},\bm{\rho}^{-}}_{u}(\bm{\lambda},\bm{\mu})$ vanishes unless
$\mu^{k}=\lambda^{k}$ when $k\in\mbb{M}$ and $\mu^{k}=\lambda^{k+1}$ when $k\not\in\mbb{M}$ for all $k=1,\dots, N$, and $\bm{\lambda}$ is a cylindric partition of periodicity $N$ and boundary profile $\mbb{M}$.
Moreover, for such $\bm{\mu}$, we have
\begin{equation}
	W^{\bm{\rho}^{+},\bm{\rho}^{-}}_{s^{N}}(\bm{\lambda},\bm{\mu})=\Phi_{\bm{\lambda}}(q,t) s^{|\bm{\lambda}|},
\end{equation}
where
\begin{equation}
	\Phi_{\bm{\lambda}}(q,t):=\prod_{i=1}^{N}J_{[\lambda^{i+1}/\lambda^{i}]}(q,t),\ \ 
	J_{[\lambda^{i+1}/\lambda^{i}]}(q,t):=
	\begin{cases}
		\psi_{\lambda^{i+1}/\lambda^{i}}(q,t), & i\in\mbb{M},\\
		\varphi_{\lambda^{i}/\lambda^{i+1}}(q,t), & i\not\in\mbb{M}.
	\end{cases}
\end{equation}
\end{lem}
\begin{proof}
The proof goes in the analogous way as the proof of \cite[Proposition 2.3]{Vuletic2009}.
Since for each $k=1,\dots, N$, either $\rho^{+}_{k}$ or $\rho^{-}_{k}$ is the zero specialization, the weight vanishes unless $\mu^{k}=\lambda^{k}$ if $k\in\mbb{M}$ and $\mu^{k}=\lambda^{k+1}$ if $k\not\in\mbb{M}$.
The formula \cite[Chapter VI, (7.14) and (7.14$^{\prime}$)]{Macdonald1995} for the single variable specializations of the Macdonald symmetric functions
verifies that the weight $W^{\bm{\rho}^{+},\bm{\rho}^{-}}_{u}(\bm{\lambda},\bm{\mu})$ is nonzero only if $\bm{\lambda}\in \mcal{CP}(N,\mbb{M})$.
Noting that, we defined the specializations $\rho^{+}_{0}$ and $\rho^{-}_{N}$ at the {\it boundary} separately from the others, for $k=1,\dots, N-1$, we have
\begin{align}
	Q_{\lambda^{k}/\mu^{k}}(\rho^{-}_{k})P_{\lambda^{k+1}/\mu^{k}}(\rho^{+}_{k})=
	\begin{cases}
		\chi[\mu^{k}=\lambda^{k}]\psi_{\lambda^{k+1}/\lambda^{k}}(q,t)s^{-k(|\lambda^{k+1}|-|\lambda^{k}|)}, & k\in\mbb{M}, \\
		\chi[\mu^{k}=\lambda^{k+1}]\varphi_{\lambda^{k}/\lambda^{k+1}}(q,t)s^{k(|\lambda^{k}|-|\lambda^{k+1}|)}, & k\not\in\mbb{M},
	\end{cases}
\end{align}
and
\begin{align}
	s^{N|\mu^{N}|}Q_{\lambda^{N}/\mu^{N}}(\rho^{-}_{N})P_{\lambda^{1}/\mu^{N}}(\rho^{+}_{0})=
	\begin{cases}
		\chi[\mu^{N}=\lambda^{N}]\psi_{\lambda^{1}/\lambda^{N}}(q,t)s^{N|\lambda^{N}|}, & N\in\mbb{M}, \\
		\chi[\mu^{N}=\lambda^{1}]\varphi_{\lambda^{N}/\lambda^{1}}(q,t)s^{N|\lambda^{N}|}, & N\not\in\mbb{M}.
	\end{cases}
\end{align}
Therefore, we can see that
\begin{align}
	W^{\bm{\rho}^{+},\bm{\rho}^{-}}_{s^{N}}(\bm{\lambda},\bm{\mu})=\Phi_{\bm{\lambda}}(q,t)s^{\sum_{k=1}^{N-1}k(|\lambda^{k}|-|\lambda^{k+1}|)+N|\lambda^{N}|}=\Phi_{\bm{\lambda}}(q,t)s^{|\bm{\lambda}|},
\end{align}
where $\mu^{k}=\lambda^{k}$ if $k\in\mbb{M}$ and $\mu^{k}=\lambda^{k+1}$ if $k\not\in\mbb{M}$, $k=1,\dots, N$.
\end{proof}

It is obvious from the definition of the weight $W^{\bm{\rho}^{+},\bm{\rho}^{-}}_{u}(\bm{\lambda},\bm{\mu})$ that
\begin{equation}
	\sum_{\bm{\lambda}\in\mcal{CP}(N,\mbb{M})}\Phi_{\bm{\lambda}}(q,t)s^{|\bm{\lambda}|}=\sum_{\bm{\lambda},\bm{\mu}\in\mbb{Y}^{N}}W_{s^{N}}^{\bm{\rho}^{+},\bm{\rho}^{-}}(\bm{\lambda},\bm{\mu})=\Pi_{q,t;s^{N}}(\bm{\rho}^{+};\bm{\rho}^{-}).
\end{equation}
Now the right hand side can be further computed by means of Proposition \ref{prop:Cauchy_identity} to be
\begin{align}
	\Pi_{q,t;s^{N}}(\bm{\rho}^{+};\bm{\rho}^{-})=&\frac{1}{(s^{N};s^{N})_{\infty}}\prod_{\substack{k<l\\ k\in\mbb{M},\ l\not\in\mbb{M}}}\frac{(ts^{l-k};q,s^{N})_{\infty}}{(s^{l-k};q,s^{N})_{\infty}}\prod_{\substack{k>l \\ k\in\mbb{M},\ l\not\in\mbb{M}}}\frac{(ts^{N-k+l};q,s^{N})_{\infty}}{(s^{N-k+l};q,s^{N})_{\infty}} \\
	=&\frac{1}{(s^{N};s^{N})}\prod_{k\in\mbb{M},\ l\not\in\mbb{M}}\frac{(ts^{[l-k]};q,s^{N})_{\infty}}{(s^{[l-k]};q,s^{N})_{\infty}}. \notag
\end{align}

It remains to show the following property.
\begin{lem}
For every $\bm{\lambda}\in\mcal{CP}(N,\mbb{M})$, we have
\begin{equation}
\label{eq:coincidence_coefficients_of_weights}
	F_{\bm{\lambda}}(q,t)=\Phi_{\bm{\lambda}}(q,t).
\end{equation}
\end{lem}
\begin{proof}
This is the analogy of \cite[Proposition 2.4]{Vuletic2009} for the case of plane partitions.
Note that the periodicity of a sequence of partitions do not make any difference regarding Eq.(\ref{eq:coincidence_coefficients_of_weights}),
and the only difference is that we impose a boundary profile.
Therefore, Eq.~(\ref{eq:coincidence_coefficients_of_weights}) is proved by modifying the proof of \cite[Proposition 2.4]{Vuletic2009} taking care of the effects of a boundary profile.
\end{proof}

\subsection{Limiting cases}
Let us investigate some limiting cases.
Obviously, at the Schur-limit $q\to t$, the weight $F_{\bm{\lambda}}(q,t)$ reduces to the unity and our formula recovers the MacMahon formula for cylindric partitions presented in \cite{Borodin2007}.
\begin{cor}
Let $N\in\mbb{N}$ and let $\mbb{M}$ be a boundary profile. Then
\begin{equation}
	\sum_{\bm{\lambda}\in\mcal{CP}(N,\mbb{M})}s^{|\bm{\lambda}|}=\frac{1}{(s^{N};s^{N})_{\infty}}\prod_{k\in\mbb{M},l\not\in\mbb{M}}\frac{1}{(s^{[l-k]};s^{N})_{\infty}}.
\end{equation}
\end{cor}

The generalized MacMahon formula for plane partitions \cite{Vuletic2009} is recovered at the infinite period limit $N\to\infty$.
To see this, we translate a boundary profile $\mbb{M}$ in $\mbb{Z}$ and understand $\mbb{M}\subset \set{-N,-N+1,\dots, N-2, N-1}$.
Therefore, it is a boundary profile of cylindric partitions of periodicity $2N$.
In particular, if we take $\mbb{M}=\set{-N,-N+1,\dots,-1}$, we have
\begin{equation}
	\sum_{\bm{\lambda}\in\mcal{CP}(2N,\mbb{M})}F_{\bm{\lambda}}(q,t)s^{|\bm{\lambda}|}=\frac{1}{(s^{2N};s^{2N})_{\infty}}\prod_{k=-N}^{-1}\prod_{l=0}^{N-1}\frac{(ts^{l-k};q,s^{2N})_{\infty}}{(s^{l-k};q,s^{2N})_{\infty}}.
\end{equation}
At the limit $N\to\infty$, the collection $\mcal{CP}(2N,\mbb{M})$ of cylindric partitions approaches the collection of plane partitions $\mcal{P}$
and, assuming $s\in (0,1)$, the right hand side converges to
\begin{equation}
	\prod_{k=-\infty}^{-1}\prod_{l=0}^{\infty}\frac{(ts^{l-k};q)_{\infty}}{(s^{l-k};q)_{\infty}}=\prod_{n=1}^{\infty}\left(\frac{(ts^{n};q)_{\infty}}{(s^{n};q)_{\infty}}\right)^{n}.
\end{equation}

To describe the Hall--Littlewood limit $q\to 0$, we introduce some notions.
Let us fix a cylindric partition $\bm{\lambda}\in\mcal{CP}(N,\mbb{M})$ of periodicity $N$ and boundary profile $\mbb{M}$.
For a position $(k,j)\in [1,N]\times \mbb{N}$, the level $h(k,j)$ is defined by
\begin{equation}
	h(k,j)=\mrm{min}\set{h\in\mbb{N}|\lambda^{k}_{j+h}<\lambda^{k}_{j}}.
\end{equation}
We also say that a position $(k,j)$ is adjacent to $(k+1,j-\chi [k\not\in\mbb{M}])$ and $(k+1,j+\chi [k\in\mbb{M}])$. Then the support of $\bm{\lambda}$, the collection of positions with nonzero entries, is decomposed into a disjoint union of connected components, each of which consists of positions of the same level. Obviously, such a connected component contains at most $N$ boxes and, if it indeed consists of $N$ boxes, it winds around the cylinder.
In contrast to such a {\it global} connected component, we call a connected component of a fixed level $B$ is {\it local} if $|B|<N$.
Now, to the cylindric partition $\bm{\lambda}$, we associate a weight
\begin{equation}
	A_{\bm{\lambda}}(t)=\prod_{B\subset\mrm{Supp}(\bm{\lambda}):\mrm{local}}(1-t^{h(B)}),
\end{equation}
where the product runs over local connected components of fixed levels and $h(B)$ is the level of boxes constituting $B$.
The following proposition obtained in \cite{CorteelSaveliefVuletic2011} is the Hall--Littlewood limit of Theorem \ref{thm:MacMahon_cylindric_partitions}.
\begin{prop}
\label{prop:MacMahon_HL}
Let $N\in\mbb{N}$ and let $\mbb{M}$ be a boundary profile. Then
\begin{equation}
	\sum_{\bm{\lambda}\in\mcal{CP}(N,\mbb{M})}A_{\bm{\lambda}}(t)s^{|\bm{\lambda}|}=\frac{1}{(s^{N};s^{N})_{\infty}}\prod_{k\in\mbb{M},l\not\in\mbb{M}}\frac{(ts^{[l-k]};s^{N})_{\infty}}{(s^{[l-k]};s^{N})_{\infty}}.
\end{equation}
\end{prop}
\begin{proof}
Our goal is to show that $F_{\bm{\lambda}}(0,t)=A_{\bm{\lambda}}(t)$ for all $\bm{\lambda}\in\mcal{CP}(N,\mbb{M})$.
Fix $\bm{\lambda}$ and assume that $(k,j)\in [1,N]\times\mbb{N}$ is of level $h$ and belongs to a connected component $B$ of the same level. The weight $F_{\bm{\lambda}}(k,j;q,t)$ is decomposed into three parts:
\begin{align}
	F_{\bm{\lambda}}(k,j;q,t)=&\prod_{m=0}^{h-2}\frac{f(q^{\lambda_{1}-\lambda_{m+1}}t^{m})f(q^{\lambda_{1}-\lambda_{m+2}}t^{m})}{f(q^{\lambda_{1}-\mu_{m+1}}t^{m})f(q^{\lambda_{1}-\nu_{m+1}}t^{m})} \\
	&\times \frac{f(q^{\lambda_{1}-\lambda_{h}}t^{h-1})f(q^{\lambda_{1}-\lambda_{h+1}}t^{h-1})}{f(q^{\lambda_{1}-\mu_{h}}t^{h-1})f(q^{\lambda_{1}-\nu_{h}}t^{h-1})} \notag \\
	&\times \prod_{m=0}^{h-2}\frac{f(q^{\lambda_{1}-\lambda_{m+1}}t^{m})f(q^{\lambda_{1}-\lambda_{m+2}}t^{m})}{f(q^{\lambda_{1}-\mu_{m+1}}t^{m})f(q^{\lambda_{1}-\nu_{m+1}}t^{m})}. \notag
\end{align}
The partitions $\lambda,\mu,\nu$ associated with $(k,j)$ are defined in Eqs.~(\ref{eq:partition_associated_to_a_box_1})--(\ref{eq:partition_associated_to_a_box_3}).
The first part is identical to unity because, by definition of the level, we have
\begin{equation}
	\lambda_{1}\ge \mu_{1}\ge \cdots \ge \mu_{h-1}\ge \lambda_{h}=\lambda_{1},\ \ \lambda_{1}\ge \nu_{1}\ge \cdots \ge \nu_{h-1}\ge \lambda_{h}=\lambda_{1}.
\end{equation}
Since $f(u)\to 1-tu$ at the Hall--Littlewood limit $q\to 0$, the third parts converge to unity at that limit. Only the second part has nontrivial limit so that
\begin{equation}
	F_{\bm{\lambda}}(k,j;0,t)
	=\frac{1-t^{h}}{(1-t^{h})^{\chi [\lambda_{1}=\mu_{h}]}(1-t^{h})^{\chi [\lambda_{1}=\nu_{h}]}}=(1-t^{h})^{1-c(k,j)},
\end{equation}
where we defined
\begin{equation}
	c(k,j)=\chi [(k-1,j+\chi [k-1\not\in\mbb{M}])\in B]+\chi [(k+1,j+\chi [k\in\mbb{M}])\in B].
\end{equation}
Notice that
\begin{equation}
	\sum_{(k,j)\in B}c(k,j)=
	\begin{cases}
		|B|-1 & B:\mrm{local}, \\
		|B|=N & B:\mrm{global}.
	\end{cases}
\end{equation}
Therefore, we have
\begin{equation}
	\prod_{(k,j)\in B}F_{\bm{\lambda}}(k,j;0,t)=
	\begin{cases}
		1-t^{h} & B:\mrm{local}, \\
		1 & B:\mrm{global},
	\end{cases}
\end{equation}
which proves $F_{\bm{\lambda}}(0,t)=A_{\bm{\lambda}}(t)$.
\end{proof}

We say that a cylindric partition $\bm{\lambda}\in\mcal{CP}(N,\mbb{M})$ is a strict cylindric partition if there is no local connected component of level larger than 1. We denote $\mcal{SCP}(N,\mbb{M})$ is the collection of strict cylindric partitions in $\mcal{CP}(N,\mbb{M})$. Since the weight $A_{\bm{\lambda}}(t)$ is defined multiplicatively, it vanishes at $t=-1$ if $\bm{\lambda}$ contains a local connected component of level larger than 1. Therefore, we have
\begin{equation}
	A_{\bm{\lambda}}(-1)=\chi [\bm{\lambda}\in\mcal{SCP}(N,\mbb{M})] 2^{k(\bm{\lambda})},
\end{equation}
where $k(\bm{\lambda})$ is the number of local connected components (of level 1) in the support of $\bm{\lambda}$.
In \cite{FodaWheeler2007,Vuletic2007}, the authors found the shifted MacMahon formula for strict plane partitions. The following corollary of Proposition \ref{prop:MacMahon_HL} is the analogue of their formula for cylindric partitions.
\begin{cor}
Let $N\in\mbb{N}$ and let $\mbb{M}$ be a boundary profile. Then
\begin{equation}
	\sum_{\bm{\lambda}\in\mcal{SCP}(N,\mbb{M})}2^{k(\bm{\lambda})}s^{|\bm{\lambda}|}=\frac{1}{(s^{N};s^{N})_{\infty}}\prod_{k\in\mbb{M},l\not\in\mbb{M}}\frac{(-s^{[l-k]};s^{N})_{\infty}}{(s^{[l-k]};s^{N})_{\infty}}.
\end{equation}
\end{cor}

\section{Trace of Macdonald refined topological vertices}
\label{app:another_trace_topolgical_vertex}
A topological vertex is a combinatorial object that counts the number of plane partitions with prescribed asymptotic partitions.
In \cite{FodaWu2017}, the authors proposed a generalization of a topological vertex called a Macdonald refined topological vertex  that unifies the refinement of a topological vertex \cite{AwataKanno2005, IqbalKozcazVafa2009, AwataKanno2009, AwataFeiginShiraishi2012} and the Macdonald deformation \cite{Vuletic2009}.
For a triple of partitions $\lambda$, $\mu$ and $\nu$, the corresponding Macdonald refined topological vertex is
\begin{equation}
	\mcal{V}_{\lambda\mu\nu}(x,y;q,t)
	=\prod_{s\in\nu}\frac{(tx^{l_{\nu}(s)+1}y^{a_{\nu}(s)};q)_{\infty}}{(x^{l_{\nu}(s)+1}y^{a_{\nu}(s)};q)_{\infty}}\sum_{\eta\in\mbb{Y}}P_{\lambda/\eta}(y^{\rho-1}x^{-\pr{\nu}};q,t)Q_{\mu/\eta}(x^{\rho}y^{-\nu};q,t),
\end{equation}
where the specializations are defined by $y^{\rho-1}x^{-\pr{\nu}}:x_{i}\mapsto y^{i-1}x^{-\pr{\nu}_{i}}$ and $x^{\rho}y^{-\nu}:x_{i}\mapsto x^{i}y^{-\nu_{i}}$, $i\ge 1$.
When we set $q=t$, it reduces to the refined topological vertex and when we further set $x=y$, it reduces to the topological vertex \cite{AganagicKlemmMarinoVafa2005}.
In \cite{BryanKoolYoung2018}, the authors presented several trace identities for topological vertices.
A generalization of one of their results to a Macdonald refined topological vertex reads as follows:
\begin{thm}
\label{thm:trace_topological_vertex}
For any partition $\nu\in\mbb{Y}$, we have
\begin{align}
	&\sum_{\lambda\in\mbb{Y}}u^{|\lambda|}\mcal{V}_{\lambda\lambda\nu}(x,y;q,t) \\
	&=\frac{1}{(u;u)_{\infty}}\prod_{s\in\nu}\frac{(tx^{l_{\nu}(s)+1}y^{a_{\nu}(s)};q)_{\infty}}{(x^{l_{\nu}(s)+1}y^{a_{\nu}(s)};q)_{\infty}}\prod_{s\in\mrm{re}(\nu)}\frac{(tx^{-l_{\nu}(s)}y^{-a_{\nu}(s)-1};q,u)_{\infty}}{(x^{-l_{\nu}(s)}y^{-a_{\nu}(s)-1};q,u)_{\infty}} \notag \\
	&\hspace{20pt}\times  \prod_{s\in\del_{\mrm{B}}\nu}\frac{(tx^{-l_{\mrm{re}(\nu)}(s)+1}y^{j(s)-1};q,u,x)_{\infty}}{(x^{-l_{\mrm{re}(\nu)}(s)+1}y^{j(s)-1};q,u,x)_{\infty}}\prod_{s\in\del_{\mrm{R}}\nu}\frac{(tx^{i(s)}y^{a_{\mrm{re}(\nu)}(s)};q,u,y)_{\infty}}{(x^{i(s)}y^{a_{\mrm{re}(\nu)}(s)};q,u,y)_{\infty}} \notag \\
	&\hspace{20pt}\times \frac{(tx^{\ell (\nu)+1}y^{\nu_{1}};q,u,x,y)_{\infty}}{(x^{\ell (\nu)+1}y^{\nu_{1}};q,u,x,y)_{\infty}}. \notag
\end{align}
Here $\mrm{re}(\nu)=(\nu_{1}^{\ell (\nu)})$ is the rectangular envelop of $\nu$ and $\del_{\mrm{R}}\nu=\set{(i,\nu_{i}):i=1,\dots,\ell(\nu)}$ and $\del_{\mrm{B}}\nu=\set{(\pr{\nu}_{j},j):j=1,\dots,\nu_{1}}$ are the right and bottom boundaries of $\nu$.
\end{thm}

\begin{proof}
Let us write a Macdonald refined topological vertex in terms of a matrix element of vertex operators so that
\begin{equation}
	\mcal{V}_{\lambda\mu\nu}(x,y;q,t)=\prod_{s\in\nu}\frac{(tx^{l_{\nu}(s)+1}y^{a_{\nu}(s)};q)_{\infty}}{(x^{l_{\nu}(s)+1}y^{a_{\nu}(s)};q)_{\infty}}
	\braket{P_{\lambda}(q,t)|\Gamma (x^{-\pr{\nu}}y^{\rho -1})_{-}\Gamma (x^{\rho}y^{-\nu})_{+}|Q_{\mu}(q,t)}.
\end{equation}
Therefore, the left hand side of Theorem \ref{thm:trace_topological_vertex} is simply
\begin{align}
	\sum_{\lambda\in\mbb{Y}}u^{|\lambda|}\mcal{V}_{\lambda\lambda\nu}(x,y;q,t)
	&=\prod_{s\in\nu}\frac{(tx^{l_{\nu}(s)+1}y^{a_{\nu}(s)};q)_{\infty}}{(x^{l_{\nu}(s)+1}y^{a_{\nu}(s)};q)_{\infty}}\mrm{Tr}_{\mcal{F}}\left(u^{D}\Gamma (x^{\rho}y^{-\nu})_{+}\Gamma (x^{-\pr{\nu}}y^{\rho-1})_{-}\right) \\
	&=\prod_{s\in\nu}\frac{(tx^{l_{\nu}(s)+1}y^{a_{\nu}(s)};q)_{\infty}}{(x^{l_{\nu}(s)+1}y^{a_{\nu}(s)};q)_{\infty}}\Pi_{q,t;u}(x^{\rho}y^{-\nu};x^{-\pr{\nu}}y^{\rho-1}). \notag
\end{align}
For those specializations, we have
\begin{align}
	p_{n}(x^{\rho}y^{-\nu})&=\sum_{i=1}^{\ell(\nu)}(x^{i}y^{-\nu_{i}})^{n}+\frac{x^{n(\ell(\nu)+1)}}{1-x^{n}},\\
	p_{n}(x^{-\pr{\nu}}y^{\rho-1})&=\sum_{i=1}^{\nu_{1}}(x^{-\pr{\nu}_{i}}y^{i-1})^{n}+\frac{y^{n\nu_{1}}}{1-y^{n}}
\end{align}
for $n>0$.
Therefore,
\begin{align}
	\tilde{\Pi}_{q,t;u}(x^{\rho}y^{-\nu};x^{-\pr{\nu}}y^{\rho-1})
	=&\prod_{i=1}^{\ell(\nu)}\prod_{j=1}^{\nu_{1}}\frac{(tx^{i-\pr{\nu}_{j}}y^{j-\nu_{i}-1};q,u)_{\infty}}{(x^{i-\pr{\nu}_{j}}y^{j-\nu_{i}-1};q,u)_{\infty}} \\
	&\times \prod_{i=1}^{\ell(\nu)}\frac{(tx^{i}y^{\nu_{1}-\nu_{i}};q,u,y)_{\infty}}{(x^{i}y^{\nu_{1}-\nu_{i}};q,u,y)_{\infty}} \prod_{j=1}^{\nu_{1}}\frac{(tx^{-\pr{\nu}_{j}+\ell(\nu)+1}y^{j-1};q,u,x)_{\infty}}{(x^{-\pr{\nu}_{j}+\ell(\nu)+1}y^{j-1};q,u,x)_{\infty}} \notag \\
	&\times \frac{(tx^{\ell(\nu)+1}y^{\nu_{1}};q,u,x,y)_{\infty}}{(x^{\ell(\nu)+1}y^{\nu_{1}};q,u,x,y)_{\infty}}. \notag
\end{align}
For a box $s=(i,j)$, we write $i(s):=i$ and $j(s)=j$. Recall that the arm length and leg length are defined as $a_{\nu}(s)=\nu_{i}-j$ and $l_{\nu}(s)=\pr{\nu}_{j}-i$.
The product in the first line is understood as the one over the rectangular envelop $\mrm{re}(\nu)=(\nu_{1}^{\ell(\nu)})$ of the partition $\nu$.
The products in the second line are the one over the right boundary $\del_{\mrm{R}}\nu=\set{(i,\nu_{i}):i=1,\dots,\ell(\nu)}$ and the one over the bottom boundary $\del_{\mrm{B}}\del=\set{(\pr{\nu}_{j},j):j=1,\dots,\nu_{1}}$.
Therefore, we find the desired result.
\end{proof}

When $\nu=\emptyset$, the formula gets simpler.
\begin{cor}
We have
\begin{equation}
	\sum_{\lambda\in\mbb{Y}}u^{|\lambda|}\mcal{V}_{\lambda\lambda\emptyset}(x,y;q,t)=\frac{(tx;q,u,x,y)_{\infty}}{(x;q,u,x,y)_{\infty}}.
\end{equation}
\end{cor}

To write another formula regarding the trace of a Macdonald refined topological vertex, we write the set of nonnegative signatures as
\begin{equation}
	\mrm{Sign}^{+}=\bigsqcup_{N\ge 0}\mrm{Sign}^{+}_{N},\ \ \mrm{Sign}^{+}_{N}=\Set{\lambda=(\lambda_{1},\dots,\lambda_{N})\in\mbb{Z}^{N}:\lambda_{1}\ge \cdots\ge\lambda_{N}\ge 0}.
\end{equation}
We set $\ell (\lambda)=N$ if $\lambda\in \mrm{Sign}^{+}_{N}$.
Note that, in contrast to a partition, $\lambda_{\ell(\lambda)}$ can be zero.
As in the case for a partition, we also write $m_{r}(\lambda)=\#\set{i:\lambda_{i}=r}$ and $|\lambda|=\sum_{i}\lambda_{i}$.

\begin{thm}
Let $\tilde{q}$ and $\tilde{t}$ be parameters possibly different from other ones $q,t,x,y$.
We have the following identity:
\begin{align}
	&\sum_{\lambda\in\mbb{Y}}u^{|\lambda|}\mcal{V}_{\lambda\lambda\emptyset}(x,y;q,t)\frac{\mcal{V}_{\square\emptyset\pr{\lambda}}(q,t;\tilde{q},\tilde{t})}{\mcal{V}_{\emptyset\emptyset\pr{\lambda}}(q,t;\tilde{q},\tilde{t})} \\
	&=\frac{1}{1-t}\frac{(u;u)_{\infty}(tq^{-1}u;u)_{\infty}}{(tu;u)_{\infty}(q^{-1}u;u)_{\infty}}\frac{(tx;q,u,x,y)_{\infty}}{(x;q,u,x,y)_{\infty}} \notag \\
	&\hspace{20pt}\times \sum_{N=0}^{\infty}t^{-N}\sum_{\lambda,\mu\in \mrm{Sign}^{+}_{N}}\prod_{r\ge 0}\frac{(t;u)_{m_{\lambda}(r)}(t;u)_{m_{\mu}(r)}}{(u;u)_{m_{\lambda}(r)}(u;u)_{m_{\mu}(r)}}x^{|\lambda|-N}y^{|\mu|}. \notag
\end{align}
\end{thm}

\begin{proof}
First, notice that
\begin{equation}
	\frac{\mcal{V}_{\square\emptyset\lambda}(q,t;\tilde{q},\tilde{t})}{\mcal{V}_{\emptyset\emptyset\lambda}(q,t;\tilde{q},\tilde{t})}=P_{\square}(q^{-\pr{\lambda}}t^{\rho-1};\tilde{q},\tilde{t})=e_{1}(q^{-\pr{\lambda}}t^{\rho-1})=\mcal{E}^{\prime}_{1}(\lambda^{\prime}).
\end{equation}
Therefore, it follows that
\begin{align}
	\sum_{\lambda\in\mbb{Y}}u^{|\lambda|}\mcal{V}_{\lambda\lambda\emptyset}(x,y;q,t)\frac{\mcal{V}_{\square\emptyset\pr{\lambda}}(q,t;\tilde{q},\tilde{t})}{\mcal{V}_{\emptyset\emptyset\pr{\lambda}}(q,t;\tilde{q},\tilde{t})}=\mrm{Tr}_{\mcal{F}}\left(u^{D}\Gamma (x^{\rho})_{+}\mcal{O}(\mcal{E}^{\prime}_{1})\Gamma (y^{\rho-1})_{-}\right),
\end{align}
which has been already computed in Subsection \ref{subsect:second_series}. Hence, we have
\begin{align}
	&\sum_{\lambda\in\mbb{Y}}u^{|\lambda|}\mcal{V}_{\lambda\lambda\emptyset}(x,y;q,t)\frac{\mcal{V}_{\square\emptyset\pr{\lambda}}(q,t;\tilde{q},\tilde{t})}{\mcal{V}_{\emptyset\emptyset\pr{\lambda}}(q,t;\tilde{q},\tilde{t})}\\
	&=\frac{1}{1-t}\int\frac{dz}{2\pi\sqrt{-1}z}\Pi_{q,t;u}(x^{\rho};y^{\rho-1})\frac{(u;u)_{\infty}(tq^{-1}u;u)_{\infty}}{(tu;u)_{\infty}(q^{-1}u;u)_{\infty}} \notag\\
	&\hspace{70pt}\times \exp\left(-\sum_{n>0}\frac{1-t^{-n}}{1-u^{n}}(t/q)^{n/2}\frac{p_{n}(x^{\rho})}{n}z^{n}\right) \notag\\
	&\hspace{70pt}\times \exp\left(\sum_{n>0}\frac{1-t^{n}}{1-u^{n}}(t/q)^{n/2}\frac{p_{n}(y^{\rho-1})}{n}z^{-n}\right) \notag \\
	&=\frac{1}{1-t}\frac{(u;u)_{\infty}(tq^{-1}u;u)_{\infty}}{(tu;u)_{\infty}(q^{-1}u;u)_{\infty}}\frac{(tx;q,u,x,y)_{\infty}}{(x;q,u,x,y)_{\infty}} \notag \\
	&\hspace{20pt}\times \int \frac{dz}{2\pi\sqrt{-1}z}\frac{((t/q)^{1/2}xz;u,x)_{\infty}}{((t/q)^{1/2}t^{-1}xz;u,x)_{\infty}}\frac{((t/q)^{1/2}tz^{-1};u,y)_{\infty}}{((t/q)^{1/2}z^{-1};u,y)_{\infty}}.\notag
\end{align}
Now we notice the following lemma:
\begin{lem}
We have
\begin{equation}
	\frac{(az;x,y)_{\infty}}{(z;x,y)_{\infty}}=\sum_{\lambda\in\mrm{Sign}^{+}}\prod_{r\ge 0}\frac{(a;x)_{m_{\lambda}(r)}}{(x;x)_{m_{\lambda}(r)}}y^{|\lambda|}z^{\ell (\lambda)}.
\end{equation}
\end{lem}
\begin{proof}
It follows from a direct computation. Using the $q$-binomial theorem, we have
\begin{equation}
	\frac{(az;x,y)_{\infty}}{(z;x,y)_{\infty}}=\prod_{r=0}^{\infty}\frac{(y^{r}az;x)_{\infty}}{(y^{r}z;x)_{\infty}}
	=\prod_{r=0}^{\infty}\sum_{n=0}^{\infty}\frac{(a;x)_{n}}{(x;x)_{n}}y^{rn}z^{n}.
\end{equation}
Commuting the sum and the product, we find a sum over nonnegative signatures.
\end{proof}
This lemma allows us to evaluate the constant term to obtain the desired result.
\end{proof}

\bibliographystyle{alpha}
\bibliography{mac_process}
\end{document}